\documentclass[11pt,reqno]{amsart}
\usepackage{amsmath,amssymb,geometry,color,tikz-cd}
\usepackage[backref,pagebackref,pdftex,hyperindex]{hyperref}
\usepackage[alphabetic]{amsrefs}
\usepackage{amssymb}
\usepackage{bbm}
\usepackage{enumitem}
\usepackage{upgreek}

\newtheorem{proposition}{Proposition}[section]
\newtheorem{theorem}[proposition]{Theorem}
\newtheorem{lemma}[proposition]{Lemma}
\newtheorem{corollary}[proposition]{Corollary}

\theoremstyle{definition}
\newtheorem{remark}[proposition]{Remark}
\newtheorem{definition}[proposition]{Definition}
\newtheorem{example}[proposition]{Example}
\newtheorem{question}[proposition]{Question}

\title{Delta invariants of projective bundles and projective cones of Fano type}

\author{Kewei Zhang}
\address{School of Mathematical Sciences, Beijing Normal University, Beijing, 100875, China.}
\email{kwzhang@pku.edu.cn}

\author{Chuyu Zhou}
\address{\'Ecole Polytechnique F\'ed\'erale de Lausanne (EPFL), MA C3 615, Station 8, 1015 Lausanne, Switzerland}
\email{chuyu.zhou@epfl.ch}

\date{} 

\newcommand{\Fut}{{\rm{Fut}}}

\newcommand{\ord}{{\rm {ord}}}

\newcommand{\vol}{{\rm {vol}}}

\newcommand{\Spec}{{\rm {Spec}}}
\newcommand{\Proj}{{\rm{Proj}}}

\newcommand{\Pic}{{\rm {Pic}}}
\newcommand{\Cl}{{\rm {Cl}}}
\newcommand{\dt}{{\rm {dt}}}

\newcommand{\idd}{\sqrt{-1}\partial\bar{\partial}}
\newcommand{\Ric}{{\operatorname{Ric}}}

\newcommand{\bA}{\mathbb{A}}

\newcommand{\bC}{\mathbb{C}}

\newcommand{\bN}{\mathbb{N}}

\newcommand{\bP}{\mathbb{P}}
\newcommand{\bQ}{\mathbb{Q}}
\newcommand{\bR}{\mathbb{R}}

\newcommand{\bZ}{\mathbb{Z}}

\newcommand{\mD}{\mathcal{D}}

\newcommand{\mL}{\mathcal{L}}

\newcommand{\mO}{\mathcal{O}}

\newcommand{\mS}{\mathcal{S}}

\newcommand{\mV}{\mathcal{V}}

\newcommand{\tF}{\tilde{F}}
\newcommand{\tY}{\tilde{Y}}
\newcommand{\tX}{\tilde{X}}

\newcommand{\tV}{\tilde{V}}

\begin{document}

\begin{abstract}
In this paper, we will give a precise formula to compute delta invariants of projective bundles and projective cones of Fano type.
\end{abstract}

\maketitle
\tableofcontents

\section{Introduction}

\subsection{Motivation}

Given an arbitrary Fano manifold $V$, it is often the case that $V$ does not admit any K\"ahler--Einstein (KE) metric. But still, $V$ always admits twisted KE or conical KE metrics. To study these metrics and their degenerations, some analytic and algebraic thresholds play important roles. For instance,
the greatest Ricci lower bound $\beta(V)$ of Tian \cite{Tia92} measures how far $V$ is away from a KE manifold. As shown in \cite{BBJ18,CRZ19}, $\beta(V)$ is equal to the algebraic $\delta$-invariant introduced in \cite{FO18,BJ20}, which serves as the right threshold for $V$ to be Ding-stable (cf. \cite{BoJ18,BBJ18,BLZ19}).

So an important problem in algebraic geometry is to compute the $\delta$-invariants of Fano varieties. Although some efforts to tackle this problem have been made in low dimensions (see e.g. \cite{PW18,CRZ19,CZ19}), overall this is still a highly nontrivial problem. However as we shall see in this paper, when the variety enjoys certain symmetry, the difficulty could be substantially reduced.

More precisely, we will investigate projective bundles and projective cones of Fano type. These objects enjoy a natural $\bC^*$-action in the fiber direction. On the analytic side, this torus action allows us to carry out the momentum construction due to Calabi, using which we can control the greatest Ricci lower bound from below. On the algebraic side, by using this torus action we can easily bound delta invariant from above. Surprisingly, the bounds coming from these two directions coincide and hence give us the precise formula for the $\delta$-invariant. This generalizes the example considered in \cite[Section 3.1]{Sze11}. More generally we will also consider the (singular) log Fano setting and derive formulas for the corresponding $\delta$-invariants. In the log case, one can still carry out the Calabi trick as in \cite{LL19}. But we will take a purely algebraic approach, again making use of the $\bC^*$-action. Indeed, by \cite{BJ20,Golota19}, to compute $\delta$, it is enough to consider $\bC^*$-invariant divisorial valuations, which can greatly simplify the computation.

\subsection{Main results}
We work over complex number $\bC$.
We will first deal with the smooth projective bundles of Fano type, and then consider the more general log Fano case.

Let $V$ be an $n$-dimensional Fano manifold with Fano index $I(V)\geq2$. So we can find an ample line bundle $L$ such that
\begin{equation}
\label{eq:L=-1/rK_V}
L=-\frac{1}{r} K_V\text{ for some rational number $r>1$.}
\end{equation}
We put
$$
\tilde{Y}:=\bP_V(L^{-1}\oplus\mathcal{O}_V)\xrightarrow{p}V,
$$
which is the $\bP^1$-bundle over $V$ with respect to $L^{-1}\oplus \mO_V$.
Let $V_0$ denote the zero section and $V_\infty$ the infinity section of $\tilde{Y}$. Then
$$
-K_{\tilde{Y}}=p^*(-K_V)+V_0+V_\infty\sim_\bQ(r+1)V_\infty-(r-1)V_0
$$
is ample and hence $\tilde{Y}$ is an $(n+1)$-dimensional Fano manifold. We put
\begin{equation}
\label{eq:def-beta_0}
\beta_0:=\bigg(\frac{n+1}{n+2}\cdot\frac{(r+1)^{n+2}-(r-1)^{n+2}}{(r+1)^{n+1}-(r-1)^{n+1}}-(r-1)\bigg)^{-1}.
\end{equation}
Using binomial formula, one can easily verify the following elementary fact:
\begin{equation}
\label{eq:beta_0<1}
\beta_0\in(1/2,1).
\end{equation}
Our first main result is stated as follows.
\begin{theorem}\label{smoothdeltabundle}
	\label{thm:delta-Y-lower-upper-bound}
	One has
	\begin{align}
	    \delta(\tilde{Y})=\min\bigg\{\frac{\delta(V)r\beta_0}{1+\beta_0(r-1)},\beta_0\bigg\}.
	\end{align}
\end{theorem}
In particular, $\tilde{Y}$ cannot admit KE metrics.
But as we shall see in Section \ref{sec:calabi}, $\tilde{Y}$ does admit a family of twisted conical KE metrics.
When $\delta(V)\geq1/r+\beta_0(1-1/r)$ (this holds for example when $V$ is K-semistable), we deduce that
\begin{equation}
    \delta(V)=\beta_0.
\end{equation}
As we will show, in this case $V_0$ computes $\delta(\tilde{Y})$.
This generalizes the example $\tilde{Y}=Bl_1\bP^2$ treated in \cite{Sze11}. Indeed, when 
$\tilde{Y}=Bl_1\bP^2$, one has $V=\bP^1$, $n=1$ and $r=2$, so that $\delta(\tilde{Y})=\beta_0=6/7$, which agrees with the result obtained in \cite{Sze11,Li11}. In the case of $\delta(V)\leq1/r+\beta_0(1-1/r)$, Theorem \ref{thm:delta-Y-lower-upper-bound} gives
\begin{equation}
    \label{eq:delta-Y=delta-x}
    	\delta(\tilde{Y})=\frac{\delta(V)r\beta_0}{1+\beta_0(r-1)}.
\end{equation}
In this case, there always exists a prime divisor $F$ over $V$ computing $\delta(V)$ (see \cite[Theorem 6.7]{BLZ19}). This divisor naturally induces a divisor $\overline{F}$ over $\tilde{Y}$, and we will show that $\delta(\tilde{Y})$ is computed by $\overline{F}$.


\begin{remark}
    In \cite{zhuang20}, Zhuang computed the $\delta$-invariants of product spaces. In particular, let $\tilde{Y}=V\times\bP^1$ be the trivial $\bP^1$-bundle over $V$, then
$$
\delta(\tilde{Y})=\min\{\delta(V),1\}.
$$
So to some extent, Theorem \ref{thm:delta-Y-lower-upper-bound} generalizes this product formula.
\end{remark}

The smoothness assumption of $\tilde{Y}$ in Theorem \ref{thm:delta-Y-lower-upper-bound} is only used for a simpler exposition of our argument in Section \ref{sec:calabi}. As we now show, one can consider the following more general singular setting.

 Let $V$ be a normal projective variety of dimension $n$ and $L$ an ample line bundle on $V$. We define the affine cone over $V$ associated to $L$ to be
$$X:= \Spec \oplus_{k\in \bN}H^0(V,kL) ,$$
and the projective cone over $X$ associated to $L$ to be 
$$Y:=\Proj \oplus_{k,j\in \bN} H^0(V,kL)s^j.$$
It is clear that both $X$ and $Y$ are normal varieties of dimension $n+1$ (to see $X$ is a normal variety, we first focus on an integrally closed sub-ring $\oplus_{k\in \bN}H^0(V,kmL)$ where $m$ is sufficiently divisible such that $mL$ is very ample, then $\oplus_{k\in \bN}H^0(V,kL)$ is integral over $\oplus_{k\in \bN}H^0(V,kmL)$), and $Y$ is obtained by adding an infinite divisor $V_\infty$ to $X$, where $V_\infty$ is the infinite divisor on $Y$ defined by $s=0$.

We provide another viewpoint of $X$ and $Y$. Let  $p:\tilde Y:=\bP_V(L^{-1}\bigoplus\mO_V)\to V$ be the projective $\bP^1$-bundle over $V$ associated to $L^{-1}\bigoplus\mO_V$.  Let $V_\infty$ and $V_0$ be the infinite and zero sections of $p$  respectively. Then we know
there is a contraction  $\phi:\tilde{Y}\to Y$ which only contracts divisor $V_0$, and $\tilde{Y}\setminus V_0\cong Y\setminus o$, where $o$ is the cone vertex of $X$.

If the base $V$ is a $\bQ$-Fano variety and $L\sim_\bQ -\frac{1}{r}K_V$ is an ample line bundle on $V$ for some positive rational number $r$, then it is not hard to see that $\tY$ (for $r>1$) and $Y$ (for $r>0$) are $\bQ$-Fano varieties. Moreover, both $(\tilde{Y}, aV_0+bV_\infty)$ and $(Y,cV_\infty)$ are log Fano varieties for any rational $0\leq a<1,0\leq b<1,0\leq c<1$ if $r>1$,  and $1-r<a<1, 0\leq b<1, 0\leq c<1$ if $0<r\leq 1$. The main results of this paper is about the computation of delta invariants of these log Fano varieties.

\begin{theorem}\label{deltabundle}
Let V be a $\bQ$-Fano variety of dimension n and $L\sim_\bQ -\frac{1}{r}K_V$ an ample line bundle on V for some positive rational number $r$. Let $a,b$ be rational numbers such that $0\leq a<1,0\leq b<1$ if $r>1$, and $1-r<a<1, 0\leq b<1$ if $0<r\leq 1$. Then we have the following formula computing delta invariants,
$$\delta(\tilde{Y},aV_0+bV_\infty)= \min\bigg\{\frac{r\delta(V)}{\frac{n+1}{n+2}\frac{B^{n+2}-A^{n+2}}{B^{n+1}-A^{n+1}}}, \frac{1-a}{\frac{n+1}{n+2}\frac{B^{n+2}-A^{n+2}}{B^{n+1}-A^{n+1}}-A}, \frac{1-b}{B-\frac{n+1}{n+2}\frac{B^{n+2}-A^{n+2}}{B^{n+1}-A^{n+1}}}\bigg\}, $$
where $A=r-(1-a)$ and $B=r+(1-b)$. In particular, when $V$ is a Fano manifold and $a=b=0$ (in this case,  $r>1$ automatically), we have 
\begin{align*}
\delta(\tilde{Y})= \min\bigg\{\frac{r\delta(V)}{\frac{n+1}{n+2}\frac{(r+1)^{n+2}-(r-1)^{n+2}}{(r+1)^{n+1}-(r-1)^{n+1}}}, \frac{1}{\frac{n+1}{n+2}\frac{(r+1)^{n+2}-(r-1)^{n+2}}{(r+1)^{n+1}-(r-1)^{n+1}}-(r-1)}\bigg\}.
\end{align*}
\end{theorem}

Note that when $V$ is a Fano manifold, we naturally have $r\leq n+1$, e.g \cite[Chapter 5]{Kollar96}, thus one can easily have the comparison
$$\frac{1}{\frac{n+1}{n+2}\frac{(r+1)^{n+2}-(r-1)^{n+2}}{(r+1)^{n+1}-(r-1)^{n+1}}-(r-1)}\leq  \frac{1}{r+1-\frac{n+1}{n+2}\frac{(r+1)^{n+2}-(r-1)^{n+2}}{(r+1)^{n+1}-(r-1)^{n+1}}}.$$
This is precisely what we get in Theorem \ref{smoothdeltabundle}. If we write $\beta_{a,b}=\frac{1-a}{\frac{n+1}{n+2}\frac{B^{n+2}-A^{n+2}}{B^{n+1}-A^{n+1}}-A}$ and $\beta_{a,b}'=\frac{1-b}{B-\frac{n+1}{n+2}\frac{B^{n+2}-A^{n+2}}{B^{n+1}-A^{n+1}}}$, then the above formula can be simplified as 
$$\delta(\tilde{Y},aV_0+bV_\infty)= \min\bigg\{\frac{r\delta(V)}{1-a+A\beta_{a,b}}\beta_{a,b}, \beta_{a,b},\frac{(1-b)\beta_{a,b}}{(B-A)\beta_{a,b}-(1-a)}\bigg\}, $$
$$\delta(\tY)=\min\bigg\{\frac{r\delta(V)\beta_{0,0}}{1+(r-1)\beta_{0,0}}, \beta_{0,0}, \frac{\beta_{0,0}}{2\beta_{0,0}-1}\bigg\}.$$
We just note here that $\beta_{a,b}$ is the delta invariant computed by the divisor $V_0$, while $\beta_{a,b}'$ is computed by the divisor $V_\infty$, and $\frac{r\delta(V)}{1-a+A\beta_{a,b}}\beta_{a,b}$ is computed by some divisor arising from $V$. 

The following formula tells us the delta invariant of the projective cone.

\begin{theorem}\label{deltacone}
Let V be a $\bQ$-Fano variety of dimension n and $L\sim_\bQ -\frac{1}{r}K_V$ an ample line bundle on V for some positive rational number $r$, then we have the following formula computing delta invariants
$$\delta(Y,cV_\infty)= \min\bigg\{\frac{(n+2)r}{(n+1)(r+1-c)}\delta(V), \frac{(n+2)r}{(n+1)(r+1-c)}, \frac{(n+2)(1-c)}{r+1-c}\bigg\},$$
where $0\leq c<1$ is a rational number.
\end{theorem}

\begin{remark}
In fact, \cite[Remark 4.5]{Li2020} has essentially obtained above result for $0<r\leq 1$.
\end{remark}

Assume $c=0$. If $\delta(V)\geq 1$, we naturally have $r\leq n+1$, see \cite{Fuj18,Liu18}. If $\delta(V)<1$ then we have $r\delta(V)\leq n+1$ by \cite[Theorem D]{BJ20}. Hence the following result is deduced.

\begin{corollary}\label{kssdeltacone}
Let V be a $\bQ$-Fano variety of dimension n and $L\sim_\bQ -\frac{1}{r}K_V$ an ample line bundle on V for some positive rational number $r$.
\begin{enumerate}
\item If V is K-semistable, i.e. $\delta(V)\geq 1$, then $\delta(Y)$ is achieved by $V_0$, that is, $\delta(Y)=\frac{A_Y(V_0)}{S_Y(V_0)}=\frac{(n+2)r}{(n+1)(r+1)}\leq 1$,
\item If $V$ is not K-semistable, i.e. $\delta(V)<1$, then $\delta(Y)=\frac{(n+2)r}{(n+1)(r+1)}\delta(V)<1$.
\end{enumerate}
\end{corollary}

As we have mentioned in the beginning of this paper, there may not have any KE metric on an arbitrarily given Fano manifold, however, there may be conical KE metrics along some smooth divisors. 
Before we state next theorem, we first fix some notation. Let $V$  be a projective Fano manifold of dimension $n$, and $S$ is a smooth divisor on $V$ such that $S\sim_\bQ -\lambda K_V$ for some positive rational number $\lambda$. Write
$$E(V,S):=\{a\in [0,1)| \textit{$(V,aS)$ is K-semistable} \}. $$
As an application of Theorem \ref{deltacone}, we prove the following theorem on optimal angle of K-stability.

\begin{theorem}\label{thm:optimalangle}
Notation as above, suppose V and S are both K-semistable and $0<\lambda<1$, then $E(V,S)= [0,1-\frac{r}{n}]$, where $r=\frac{1}{\lambda}-1$. In particular, if $V$ and $S$ are both K-polystable, then $(V,aS)$ is K-polystable for any $a\in [0, 1-\frac{r}{n})$.
\end{theorem}

This result has been essentially known to experts (cf. \cite{LS14,LZ19,LL19} etc.), however, according to the authors' knowledge, it has not been explicitly written down. Since it can be derived by Theorem \ref{deltacone}, we just put it here and provide a complete algebraic proof. As a corollary, we directly have following result, which  provides an answer to the question posed in \cite[Remark 3.6]{Der16b}

\begin{corollary}\label{Pncase}
For the pair $(\bP^n, S_d)$ where $S_d$ is a smooth hypersurface of degree $1\leq d\leq n$. If $S_d$ is K-polystable (this is expected to be true), then we have $E(\bP^n, S_d)=[0, 1-\frac{r}{n}]$, where $r=\frac{n+1-d}{d}$. 
\end{corollary}

The paper is organized as follows. In Section \ref{preliminary}, we give a brief introduction to delta invariant and greatest lower Ricci bound, and also include some results on projective bundles and cones of Fano type. In Section \ref{section:bundle} and Section \ref{section:cone}, we prove Theorem \ref{deltabundle} and Theorem \ref{deltacone} respectively, by purely algebraic method. In Section \ref{section:optimalangle}, we give a complete proof of Theorem \ref{thm:optimalangle}. In the last section, as an application of our main results, we give some examples on computing delta invariants of some special hypersurfaces. Finally in Section \ref{analyticpart}, we use Calabi ansatz to prove a weaker version of Theorem \ref{smoothdeltabundle}.

\textbf{Acknowledgments.}
The first author would like to thank Yanir Rubinstein for bringing the example in \cite{Sze11} to his attention. C.Zhou would like to thank his advisor Prof. Chenyang Xu for his constant support and encouragement. We thank Yuchen Liu and Ziquan Zhuang for useful discussions. 
K. Zhang is supported by the China post-doctoral grant BX20190014.

\section{Preliminaries}\label{preliminary}

In this section, we will collect some fundamental results on projective bundles and projective cones, then we give a quick introduction of delta invariant which will play a central role in the subsequent contents, and finally we briefly recall the definition of the greatest Ricci lower bound. We say $(V,\Delta)$ is a log pair if $V$ is a projective normal variety and $\Delta$ is an effective $\bQ$-divisor on $V$ such that $K_V+\Delta$ is $\bQ$-Cartier. Suppose $f: W\to V$ is a proper birational mophism between normal varieties and $E$ is a prime divisor on $W$, then we define
$$A_{V,\Delta}(E):=\ord_E(K_W-f^*(K_V+\Delta)) $$
to be the log discrepancy of the divisor $E$ associated to the pair $(V,\Delta)$. The log pair $(V,\Delta)$ is called a log Fano variety if it admits klt singularities and $-(K_V+\Delta)$ is ample. If $\Delta=0$, we just say $V$ is a $\bQ$-Fano variety. For the concepts of klt singularities, please refer to \cite{KM98,Kollar13}.

\subsection{Projective bundles and Projective cones}
Throughout, $(V,L)$ will be a polarized pair where $V$ is a normal projective variety of dimension $n$ and $L$ an ample line bundle on $V$. Just as in the introduction, we fix some notation below,

$$Y:=\Proj \oplus_{k\in \bN}\oplus_{l\in \bN}H^0(V,kL)s^l,  \tY:=\bP_V(L^{-1}\oplus \mO_V),$$
$$X:=\Spec \oplus_{k\in \bN} H^0(V,kL)= Y\setminus V_\infty, \tX:=\tY\setminus V_\infty,$$
where $V_\infty$ is the infinite section of $\tY$. There is a natural contraction $\phi: \tY\to Y$ (resp. $\phi: \tX\to X$) which contracts $V_0$, where $V_0$ is the zero section of $\tY$. We just list the properties of $\tY$ and $Y$ in the following lemma.

\begin{lemma}\label{lemma 2.1}
Notation as above, let $p: \tY:=\bP_V(L^{-1}\oplus \mO_V)\to V$ be the natural projection, then we have
\begin{enumerate}
\item $V_\infty-V_0\sim p^*L, V_\infty|_{V_\infty}\sim L, V_0|_{V_0}\sim L^{-1}.$
\item $\mO_{\tilde{Y}}(1)\sim V_\infty,$
\item $\phi:\tilde{Y}\to Y$ is the blowup of vertex point $o\in X\subset Y$, and $\tilde{Y}\setminus V_0\cong Y\setminus o$, where $o$ is the cone vertex of $X$.
\item If $V$ is $\bQ$-Gorenstein, then $K_{\tilde{Y}}=p^*K_V\otimes\mO_{\tilde{Y}}(-2)\otimes p^*L\sim p^*K_V-V_\infty-V_0.$
\item If $V$ is a $\bQ$-Fano variety and $L\sim_\bQ -\frac{1}{r}K_V$ is an ample line bundle on V for some positive rational number r, then both $(\tY,aV_0+bV_\infty)$ and $(Y,cV_\infty)$ are log Fano pairs, where  $0\leq a<1,0\leq b<1,0\leq c<1$ if $r>1$, and $1-r<a<1, 0\leq b<1,0\leq c<1$ if $0<r\leq 1$.
\end{enumerate}
\end{lemma}

\begin{proof}
By the construction we see that
$$\tilde{Y}\cong\Proj_V S(L\oplus \mO_V)=\Proj_V \oplus_{i,j\in \bN} L^is^j \to \Proj \oplus_{i,j\in \bN} H^0(V, iL)s^j:=Y$$
is the natural blowup of the vertex $o\in Y$, thus $\mO_{\tilde{Y}}(1)$ is defined by $s=0$, which is the infinite section $V_\infty$, and $V_0|_{V_0}\sim L^{-1}$. It's clear that $V_\infty-V_0$ is a relative trivial line bundle for the projection map $p: \tilde{Y}\to V$, thus there is a line bundle $M$ on $V$ such that $p^*M\sim V_\infty-V_0$. Restrict it to $V_0$ we see $M\sim L$, and restrict it to $V_\infty$ we the see that $V_\infty|_{V_\infty}\sim L$. By now we complete the proof of the first three statements. The fourth one follows directly from (1) and (2). The only issue is to make sure $(Y,cV_\infty)$ are indeed log Fano varieties, specially with klt singularities. This follows from the following Theorem \ref{conesing}.
\end{proof}

Let $\Delta$ be an effective $\bQ$-divisor on $V$ such that $K_V+\Delta$ is $\bQ$-Cartier, and $\Delta_X$ and $\Delta_{\tilde{X}}$ be the corresponding extending divisors on 
$X$ and $\tilde{X}$ respectively. We have following results \cite{Kollar13}.

\begin{lemma}\label{conelemma}
Notation as above,  we have
\begin{enumerate}
\item $\Cl(\tilde{X})\cong \Cl(V).$ 
\item $\Pic(X)=0.$ 
\item $\Cl(X)\cong \Cl(V)/\bZ [L].$ 
\item $K_{\tilde{X}}+\Delta_{\tilde{X}}+V_0\sim p^*(K_V+\Delta),$ where we also use p to denote the projection $\tX\to V$.
\item We assume more that $K_X+\Delta_X$ is $\bQ$-Cartier,
then there is some rational number $r$ such that $L\sim_\bQ -\frac{1}{r}(K_V+\Delta)$. 
\item Conversely, if $K_V+\Delta$ is proportional to $L$, then $K_X+\Delta_X$ is $\bQ$-Cartier. 
\item We still assume $K_X+\Delta_X$ is $\bQ$-Cartier, then $K_{\tilde{X}}+\Delta_{\tilde{X}}+V_0=\phi^*(K_X+\Delta_X)+rV_0,$ where $r$ is as above and $r=A_{X,\Delta_X}(V_0).$
\end{enumerate}
\end{lemma}

\begin{proof}
For (1), one only needs to note that $\tilde{X}$ is an $\bA^1$-bundle over $V$.
For (2), we consider $V\cong V_0\hookrightarrow \tilde{X}\to X$, where the last arrow denoted by $\phi$ means the blowup  of the cone vertex. For any line bundle $M$ on $X$, it is pulled back to $\tilde{X}$ to be a trivial line bundle, thus $M$ is also a trivial line bundle. For (3), consider the exact sequence 
$$\bZ[V_0]\to \Cl(\tilde{X})\to \Cl(X\setminus o)\to 0,$$
 then the facts $\Cl(X\setminus o)\cong \Cl(X)$ and $V_0|_{V_0}\sim L^{-1}$ conclude the result. For (4), it is directly implied if we write down the differential form. For (5), as $m(K_X+\Delta_X)$ is a trivial line bundle for a divisible $m$ due to (2), $m(K_{X\setminus o}+\Delta_{X\setminus o})$ is also trivial, then by the exact sequence above,  one sees there is a rational number $r$ such that $p^*(K_V+\Delta)\sim_\bQ rV_0$. Hence (5) is finished by restriction. For (6), one only needs to note that $m(K_X+\Delta_X)$ is Cartier if and only if $m(K_X+\Delta_X)$ is a trivial line bundle if and only if $m(K_{X\setminus o}+\Delta_{X\setminus o})$ is a trivial line bundle, then apply the above exact sequence. For (7), we apply (5) to get a rational number $r$ such that $L\sim_\bQ -\frac{1}{r}(K_V+\Delta)$, so $K_{\tilde{X}}+\Delta_{\tilde{X}}+V_0\sim_{\bQ,\phi} rV_0$ by (4).
\end{proof}

The above lemma directly implies the following result on cone singularities, also see \cite{Kollar13}.

\begin{theorem}\label{conesing}
Let $(V,\Delta)$ be a log pair of dimension n, then
\begin{enumerate}
 \item If $(V,\Delta)$ is a log Fano pair and $L=-\frac{1}{r}(K_V+\Delta)$ is an ample line bundle on V for some rational $r>0$, then $(X,\Delta_X)$ admits klt singularities.
\item If $(V,\Delta)$ is a general type pair and $L=-\frac{1}{r}(K_V+\Delta)$ is an ample line bundle on V for some rational $r<0$, then $(X,\Delta_X)$ is not log canonical.
\item If $(V,\Delta)$ is a log canonical Calabi-Yau pair and $L$ is an ample line bundle on V, then $(X,\Delta_X)$ is log canonical.\end{enumerate}
\end{theorem}

\begin{corollary}
Let V be a projective normal variety and L an ample line bundle on V. If $X:=\Spec \oplus_{k\in bN} H^0(V,kL)$ admits klt singularities, specially, $K_X$ is $\bQ$-Cartier, then there is a positive rational number $r>0$ such that $L\sim_\bQ -\frac{1}{r}K_V$. In particular, V is a $\bQ$-Fano variety.
\end{corollary}

\subsection{Delta invariant}

In this section we assume $(V,\Delta)$ is a log Fano variety of dimension $n$. Delta invariant is introduced in \cite{FO18} to measure the singularities of the anti-canonical divisor of a $\bQ$-Fano variety, which is proved to be a powerful K-stability threshold and has led to many progresses in the field of K-stability of Fano varieties. For various concepts of K-stability please refer to \cite{Fuj19c,BJ20,BX19,BLX19} etc.. We first define $m$-basis type divisors of $(V,\Delta)$.

\begin{definition}
For a sufficiently divisible fixed natural number $m$ such that $-m(K_V+\Delta)$ is an ample line bundle, one chooses any complete basis of $H^0(V,-m(K_V+\Delta))$, say $\{s_1,s_2,...,s_{r_m} \}$, where $r_m=\dim H^0(V,-m(K_V+\Delta))$, define $D_m:=\frac{\sum_{i=1}^{r_m}{\rm div}(s_i=0)}{mr_m}$. Divisors of this form are called $m$-basis type divisors. It is clear that $D_m\sim_\bQ -(K_V+\Delta)$. For a prime divisor $E$ over $V$ (i.e. there exists a proper birational morphism from a normal variety $W\to V$ such that $E$ is a prime divisor on $W$), we define the following $S_m$-invariant and $m$-th delta invariant of $(V,\Delta)$,
$$S_{m,(V,\Delta)}(E):=\sup_{D_m}\ord_E(D_m),$$
$$ \delta_{m,(V,\Delta)}(\ord_E):=\frac{A_{(V,\Delta)}(E)}{S_{m,(V,\Delta)}(E)}.$$
\end{definition}

The following lemma is due to \cite{BJ20,FO18}.

\begin{lemma}
Notation as in above definition,
\begin{enumerate}
\item For each sufficiently divisible $m$ and a prime divisor $E$ over V, there is an m-basis type divisor $D_m$ such that $S_{m,(V,\Delta)}(E)=\ord_E(D_m)$.
\item Let m tend to infinity, then the limits in above definition indeed exist, denoted by $S_{(V,\Delta)}(E)$ and $\delta_{(V,\Delta)}(\ord_E)$.
\item $S$-invariant can be computed by the following integration,
$$S_{(V,\Delta)}(E)=\frac{1}{\vol(-(K_V+\Delta))}\int_0^{\infty}\vol(-(K_V+\Delta)-tE)\dt, $$
$$ \delta_{(V,\Delta)}(\ord_E)=\frac{A_{(V,\Delta)}(E)}{S_{(V,\Delta)}(E)}.$$
\end{enumerate}
\end{lemma}

\begin{definition}
Delta invariant of the log Fano pair $(V,\Delta)$ is defined to be
$$\delta(V,\Delta):=\inf_E\delta_{(V,\Delta)}(\ord_E)=\inf_E \frac{A_{(V,\Delta)}(E)}{S_{(V,\Delta)}(E)}, $$
where $E$ runs through all prime divisors over $V$.
\end{definition}

We sometimes leave out $(V,\Delta)$ in the subscript if there is no confusion. It is clear that the above definition applies to any $\bQ$-ample line bundle $L$ on $V$, and we respectively get $m$-basis type divisors associated to $L$ just by replacing $-(K_V+\Delta)$ by $L$. In this case, we use $S_{m,L}(E)$ and $S_L(E)$ (resp. $\delta_{m,L}(\ord_E)$ and $\delta_L(\ord_E))$ to denote the $S_m$-invariant and $S$-invariant (resp. $m$-th delta invariant and delta invariant), and $\delta(L):=\inf_E \frac{A_{(V,\Delta)}(E)}{S_L(E)}$.

The following result is well known by works \cite{BJ20,FO18}, we just state it here.

\begin{theorem}
The log Fano pair $(V,\Delta)$ is K-semistable if and only if $\delta(V,\Delta)\geq 1$.
\end{theorem}

\subsection{The greatest Ricci lower bound}
Let $V$ be a Fano manifold. Then the
greatest Ricci lower bound $\beta(V)$ of $V$ is defined to be
\begin{equation}
    \label{eq:def-beta-xi}
    \beta(V):=\sup\{\beta\in\bR\ |\ \exists\text{ K\"ahler form }\omega\in 2\pi c_1(V)\ \text{s.t. }\Ric(\omega)\geq\beta\omega \}.
\end{equation}
This invariant was first implicitly studied by Tian \cite{Tia92} and then explicitly introduced in \cite{Rub08,Rub09}. 
Recently it is shown independently by \cite{CRZ19} and \cite{BBJ18} that
\begin{equation}
\label{eq:beta-X=min-1-delta}
    \beta(X)=\min\{1,\delta(X)\}.
\end{equation}
Finally we remark that, suppose in addition that there is a semipositive $(1,1)$-current $\theta$ on $X$, then there are analogous results for the $\theta$-twisted $\delta$- and $\beta$-invariants (see \cite{BBJ18} for more information).

\section{Delta invariants of projective bundles of Fano type}\label{section:bundle}

In this section, we will prove Theorem \ref{deltabundle}. We start by the following lemma computing $\delta_{(\tY,aV_0+bV_\infty)}(\ord_{V_0})$ and $\delta_{(\tY,aV_0+bV_\infty)}(\ord_{V_\infty})$.

\begin{lemma}\label{delta0bundle}
Notation as in Theorem \ref{deltabundle}, we have
$$\delta_{(\tY,aV_0+bV_\infty)}(\ord_{V_0})= \frac{1-a}{\frac{n+1}{n+2}\frac{B^{n+2}-A^{n+2}}{B^{n+1}-A^{n+1}}-A}=:\beta_{a,b}, $$
and
$$\delta_{(\tY,aV_0+bV_\infty)}(\ord_{V_\infty})= \frac{1-b}{B-\frac{n+1}{n+2}\frac{B^{n+2}-A^{n+2}}{B^{n+1}-A^{n+1}}}=:\beta_{a,b}'. $$
\end{lemma}

\begin{proof}
It is clear that $A_{(\tY,aV_0+bV_\infty)}(V_0)=1-a, A_{(\tY,aV_0+bV_\infty)}(V_\infty)=1-b$, and by Lemma \ref{lemma 2.1},
\begin{align}\label{supV0}
-(K_{\tY}+aV_0+bV_\infty)\sim_\bQ (r+1-b)p^*L+(2-a-b)V_0.
\end{align}
Denote $\vol(t):=\vol((r+1-b)p^*L+(2-a-b-t)V_0)$, then we have
\begin{align*}
&S_{(\tY,aV_0+bV_\infty)}(V_0)\\
=&\frac{1}{\vol(-(K_{\tY}+aV_0+bV_\infty))}\int_0^{\infty}\vol(-(K_{\tY}+aV_0+bV_\infty)-tV_0)\dt\\
=&\frac{1}{\vol(t=0)}\int_0^{2-a-b}\vol((r+1-b)p^*L+(2-a-b-t)V_0)\\
=&\frac{L^n}{\vol(t=0)}\int_0^{2-a-b}\sum_{j=1}^{n+1}\binom{n+1}{j}(r+1-b)^{n+1-j}(2-a-b-t)^j(-1)^{j+1}\dt\\
=&\frac{L^n}{\vol(t=0)}\int_0^{2-a-b}-(r+1-b-(2-a-b-t))^{n+1}+(r+1-b)^{n+1}\dt\\
=&\frac{1}{B^{n+1}-A^{n+1}}\int_0^{B-A}(B^{n+1}-(A+t)^{n+1})\dt\\
=&\frac{n+1}{n+2}\frac{B^{n+2}-A^{n+2}}{B^{n+1}-A^{n+1}}-A.
\end{align*}
 Similarly we have
$$S_{(\tY,aV_0+bV_\infty)}(V_\infty) =B-\frac{n+1}{n+2}\frac{B^{n+2}-A^{n+2}}{B^{n+1}-A^{n+1}}.$$
The proof is finished.
\end{proof}

Note that the above lemma is essentially a Futaki invariant computation under a fancy $\delta$-invariant guise for the natural $\bC^*$-action on $\tilde{Y}$ (see the proof of Lemma \ref{lem:V_0-compute-beta_0}), and it tells us that $\delta(\tY,aV_0+bV_\infty)\leq \min\{\beta_{a,b},\beta_{a,b}'\}$, we next show that $\delta(\tY,aV_0+bV_\infty)\leq \frac{r\delta(V)}{\frac{n+1}{n+2}\frac{B^{n+2}-A^{n+2}}{B^{n+1}-A^{n+1}}}$. 

\begin{lemma}\label{upperbundle}
Notation as in Theorem \ref{deltabundle}, we have
$$\delta(\tY,aV_0+bV_\infty)\leq \frac{r\delta(V)}{\frac{n+1}{n+2}\frac{B^{n+2}-A^{n+2}}{B^{n+1}-A^{n+1}}}. $$
\end{lemma}

\begin{proof}
We choose sufficiently divisible $m$ such that $ma,mb$ and $mr$ are all integers, and by Lemma \ref{lemma 2.1} we have
\begin{align}\label{bundleexpansion}
-m(K_{\tilde{Y}}+aV_0+bV_\infty)\sim (mr-m+ma+j)p^*L+(2m-ma-mb-j)V_\infty+jV_0.
\end{align}
Let $\tilde{R}_m$ denote  $H^0(\tilde{Y},-m(K_{\tilde{Y}}+aV_0+bV_\infty))$, we define
$$ \tilde{R}_{m,j}:=\{s\in \tilde{R}_m|\text{$\ord_{V_0}(s)=j$}\}.$$
By (\ref{supV0}), one sees that
$$\tilde{R}_m=\bigoplus^{m(2-a-b)}_{j=0} \tilde{R}_{m,j}, $$
and by (\ref{bundleexpansion})
\begin{align}
\tilde{R}_{m,j}\cong H^0(V, (mr-m+ma+j)L) .
\end{align}
Let $E$ be a prime divisor over $V$ and $E_{\tilde{Y}}$ the natural extended divisor over $\tilde{Y}$, then we want to explore the relationship between $\delta_L(\ord_E)$ and $\delta_{(\tilde{Y},aV_0+bV_\infty)}(\ord_{E_{\tilde{Y}}})$.
As it is clear that $A_V(E)=A_{(\tilde{Y},aV_0+bV_\infty)}(E_{\tilde{Y}})$, it suffices to explore the relationship between $S_L(E)$ and $S_{(\tilde{Y},aV_0+bV_\infty)}(E_{\tilde{Y}})$.
We first construct a special $m$-basis type divisor of $-(K_{\tilde{Y}}+aV_0+bV_\infty)$. 

Write $\tilde{r}_{m,j}=\dim \tilde{R}_{m,j}=\dim H^0(V,(mr-m+ma+j)L)$, we can choose an $mr-m+ma+j$-basis type divisor $\tilde{D}_{m,j}$ of $L$ such that $\tilde{D}_{m,j}$ computes $S_{mr-m+ma+j,L}(E)$. In fact, $\tilde{D}_{m,j}$ is created by the filtration induced by $\ord_E$ on $H^0(V,(mr-m+ma+j)L)$. Then we lift $\tilde{D}_{m,j}$ to be a divisor $\tilde{\mD}_{m,j}$ over $\tilde{Y}$. By (\ref{bundleexpansion}), it is not hard to see that 
\begin{align}
\tilde{\mD}_m:=\frac{\sum_{j=0}^{m(2-a-b)}\tilde{r}_{m,j}((mr-m+ma+j)\tilde{\mD}_{m,j}+jV_0+(2m-am-bm-j)V_\infty)}{m\sum_{j=0}^{m(2-a-b)} \tilde{r}_{m,j}} 
\end{align}
 is an $m$-basis type divisor of $-(K_{\tilde{Y}}+aV_0+bV_\infty)$. Then we have 
 \begin{align}\label{maxbasisdiv}
 S_{m,(\tY,aV_0+bV_\infty)}(E_{\tilde{Y}})\geq \ord_{E_{\tilde{Y}}}(\tilde{\mD}_m)=\frac{\sum_{j=0}^{m(2-a-b)}(mr-m+ma+j)\tilde{r}_{m,j}S_{mr-m+ma+j,L}(E)}{m\sum_{j=0}^{m(2-a-b)}\tilde{r}_{m,j}}.
 \end{align}
As $m$ tends to infinity, by the computation of $\sum_{j=0}^{m(2-a-b)}j\tilde{r}_{m,j}  $ and $\sum_{j=0}^{m(2-a-b)}m\tilde{r}_{m,j}$ in the next lemma, one directly obtains that
\begin{align}\label{Smaxbasisdiv}
S_{(\tilde{Y}, aV_0+bV_\infty)}(E_{\tilde{Y}})\geq  \frac{n+1}{n+2}\frac{B^{n+2}-A^{n+2}}{B^{n+1}-A^{n+1}}\cdot S_L(E),
\end{align}
where $B=r+(1-b)$ and $A=r-(1-a)$, hence we have
\begin{align}\label{deltamaxbasisdiv}
\delta_{(\tilde{Y},aV_0+bV_\infty)}(\ord_{E_{\tilde{Y}}}) \leq \frac{1}{\frac{n+1}{n+2}\frac{B^{n+2}-A^{n+2}}{B^{n+1}-A^{n+1}}}\delta_L(\ord_E).
\end{align}
As the prime divisor $E$ is arbitrarily chosen, we have
$$\delta(\tY, aV_0+bV_\infty)\leq  \frac{1}{\frac{n+1}{n+2}\frac{B^{n+2}-A^{n+2}}{B^{n+1}-A^{n+1}}}\delta(L)=  \frac{r\delta(V)}{\frac{n+1}{n+2}\frac{B^{n+2}-A^{n+2}}{B^{n+1}-A^{n+1}}}.
$$
\end{proof}

\begin{lemma}\label{lem:computation}
Notation as in the proof of Lemma \ref{upperbundle}, we have
$$\lim_{m\to \infty}\frac{\sum_{j=0}^{m(2-a-b)}j\tilde{r}_{m,j}}{m\sum_{j=0}^{m(2-a-b)}\tilde{r}_{m,j}} =\frac{n+1}{n+2}\frac{B^{n+2}-A^{n+2}}{B^{n+1}-A^{n+1}}-A.$$
\end{lemma}

\begin{proof}
The lemma is concluded by the following two computations
\begin{align*}
 &\lim_{m\to\infty}\frac{\sum_{j=0}^{m(B-A)} jh^0(V,(mA+j)L)}{m^{n+2}/n!}
 =\lim_{m\to\infty}\sum_j\frac{j}{m}\frac{h^0(V,(mA+j)L)}{m^n/n!}\frac{1}{m}\\
=&\int_0^{B-A}\vol((A+t)L)t\dt
=L^n\cdot \bigg\{\frac{B^{n+2}-A^{n+2}}{n+2}-A\frac{B^{n+1}-A^{n+1}}{n+1} \bigg\},
\end{align*}
and 
\begin{align*}
 &\lim_{m\to\infty}\frac{\sum_{j=0}^{m(B-A)} mh^0(V,(mA+j)L)}{m^{n+2}/n!}
 =\lim_{m\to\infty}\sum_j\frac{h^0(V,(mA+j)L)}{m^n/n!}\frac{1}{m}\\
=&\int_0^{B-A}\vol((A+t)L)\dt
=L^n\cdot \frac{B^{n+1}-A^{n+1}}{n+1}.
\end{align*}
\end{proof}

Combine Lemma \ref{delta0bundle} and \ref{upperbundle}, we have the following result on upper bound. 

\begin{theorem}
Notation as in Theorem \ref{deltabundle}, we have
\begin{align}
\delta(\tilde{Y},aV_0+bV_\infty)\leq \min\bigg\{\frac{r\delta(V)}{\frac{n+1}{n+2}\frac{B^{n+2}-A^{n+2}}{B^{n+1}-A^{n+1}}}, \frac{1-a}{\frac{n+1}{n+2}\frac{B^{n+2}-A^{n+2}}{B^{n+1}-A^{n+1}}-A}, \frac{1-b}{B-\frac{n+1}{n+2}\frac{B^{n+2}-A^{n+2}}{B^{n+1}-A^{n+1}}-A}\bigg\}.  
\end{align}
In particular, when $a=b=0$ (in this case, $r>1$ automatically), we have 
\begin{align}
\delta(\tilde{Y})\leq \min\bigg\{\frac{r\delta(V)}{\frac{n+1}{n+2}\frac{(r+1)^{n+2}-(r-1)^{n+2}}{(r+1)^{n+1}-(r-1)^{n+1}}}, \frac{1}{\frac{n+1}{n+2}\frac{(r+1)^{n+2}-(r-1)^{n+2}}{(r+1)^{n+1}-(r-1)^{n+1}}-(r-1)}\bigg\}.
\end{align}
\end{theorem}

Now we turn to the converse direction. First recall that in the proof of Lemma \ref{upperbundle}, we construct a special $m$-basis type divisor $\tilde{\mD}_m$ for $-(K_{\tY}+aV_0+bV_\infty)$, however, we don't know whether it is compatible with $E_{\tY}$, so we only have $S_{m,(\tY,aV_0+bV_\infty)}(E_{\tY})\geq \ord_{E_{\tY}}(\tilde{\mD}_m)$ in (\ref{maxbasisdiv}). Once this is indeed an 
equality, so are (\ref{Smaxbasisdiv}) and (\ref{deltamaxbasisdiv}), which will lead to the final proof of Theorem \ref{deltabundle}.

\begin{lemma}\label{Tcompatible}
Notation as in the proof of Lemma \ref{upperbundle}, the m-basis divisor $\tilde{\mD}_m$ we construct is compatible with $E_{\tY}$, that is 
$$S_{m,(\tY,aV_0+bV_\infty)}(E_{\tY})= \ord_{E_{\tY}}(\tilde{\mD}_m),$$
and
$$\delta_{(\tilde{Y},aV_0+bV_\infty)}(\ord_{E_{\tilde{Y}}}) = \frac{1}{\frac{n+1}{n+2}\frac{B^{n+2}-A^{n+2}}{B^{n+1}-A^{n+1}}}\delta_L(\ord_E)=\frac{r}{\frac{n+1}{n+2}\frac{B^{n+2}-A^{n+2}}{B^{n+1}-A^{n+1}}}\delta_V(\ord_E).$$
\end{lemma}

\begin{proof}
Recall in the proof of Lemma \ref{upperbundle}, we obtained
$$\tilde{R}_m=H^0(\tY,-m(K_{\tY}+aV_0+bV_\infty))=\bigoplus_{j=0}^{m(2-a-b)}\tilde{R}_{m,j},$$
and
$$\tilde{R}_{m,j}\cong H^0(V,(mr-m+ma+j)L). $$
As $E_{\tY}$ induces a $\bC^*$-invariant divisorial valuation over $\tY$, then the filtration induced by $E_{\tY}$ on $\tilde{R}_{m}$ is compatible with the filtration induced by $E$ on $H^0(V,(mr-m+ma+j)L)$, concluded.
\end{proof}

\begin{proof}[Proof of Theorem \ref{deltabundle}]
By Lemma \ref{Tcompatible}, we let $m$ tend to infinity, then both (\ref{Smaxbasisdiv}) and (\ref{deltamaxbasisdiv}) are equalities for any prime divisor $E$ over $V$, that is 
\begin{align}\label{key}
\delta_{(\tilde{Y},aV_0+bV_\infty)}(\ord_{E_{\tilde{Y}}}) = \frac{1}{\frac{n+1}{n+2}\frac{B^{n+2}-A^{n+2}}{B^{n+1}-A^{n+1}}}\delta_L(\ord_E) \geq \frac{r\delta(V)}{\frac{n+1}{n+2}\frac{B^{n+2}-A^{n+2}}{B^{n+1}-A^{n+1}}}.
\end{align}
Let $T:=\bC^*$. As $\delta(\tilde{Y},aV_0+bV_\infty)$ can be approximated by $T$-equivariant divisorial valuations over $\tY$, see \cite[Section 4]{Golota19} or \cite[Section 7]{BJ20}, we need to deal with the $T$-invariant valuations over $\tY$ whose centers lie in $V_0$ or $V_\infty$. Assume $\tF\neq V_0$ is a $T$-invariant prime divisor over $\tY$ such that $c_{\tY}(\tF)\subset V_0$, which means the center of $\ord_{\tF}$ lies in $V_0$, then there is a positive rational number $c$ and a prime divisor $F$ over $V$ such that $r(\ord_{\tF})=c\cdot \ord_F$, where $r(\ord_{\tF})$ is the restriction of $\ord_{\tF}$ to $K(V)$. Let $F_{\tY}$ be the induced prime divisor over $\tY$ extended by $F$, then by \cite{BHJ17} $\ord_{\tF}$ is a quasimonomial valuation along $V_0$ and $F_{\tY}$ with weights $(\ord_{\tF}(V_0),c)$, that is $\ord_{\tF}=\ord_{\tF}(V_0)\cdot \ord_{V_0}+c\cdot\ord_{F_{\tY}}$, so by (\ref{bundleexpansion}) and \cite{JM12} we have
\begin{align*}
A_{(\tY,aV_0+bV_\infty)}(\tF)=&\ord_{\tF}(V_0)\cdot A_{(\tY,aV_0+bV_\infty)}(V_0)+c\cdot A_{(\tY,aV_0+bV_\infty)}(F_{\tY})\\
=&\ord_{\tF}(V_0)(1-a)+c\cdot A_{(\tY,aV_0+bV_\infty)}(F_{\tY}),
\end{align*}
and
\begin{align*}
S_{(\tY,aV_0+bV_\infty)}(\tF)=&\ord_{\tF}(V_0)\lim_{m\to \infty}\frac{\sum_{j=0}^{m(B-A)}j\tilde{r}_{m,j}}{\sum_{j=0}^{m(B-A)}m\tilde{r}_{m,j}}+c\cdot S_{(\tY,aV_0+bV_\infty)}(F_{\tY})\\
=& \ord_{\tF}(V_0)\bigg\{\frac{n+1}{n+2}\frac{B^{n+2}-A^{n+2}}{B^{n+1}-A^{n+1}}-A \bigg\}+ c\cdot S_{(\tY,aV_0+bV_\infty)}(F_{\tY}),
\end{align*}
so by Lemma \ref{delta0bundle} the following holds,

\begin{align*}
\delta_{(\tY,aV_0+bV_\infty)}(\ord_{\tF})\geq  \min\{\delta_{(\tilde{Y},aV_0+bV_\infty)}(\ord_{V_0}),\delta_{(\tY,aV_0+bV_\infty)}(\ord_{F_{\tY}})\}.   
\end{align*}

For $T$-invariant divisors whose centers lie in $V_\infty$ but not equal to $V_\infty$, the analysis is similar. One can write
$$\ord_{\tF}=\ord_{\tF}(V_\infty)\cdot \ord_{V_\infty}+c\cdot\ord_{F_{\tY}} $$ 
for some positive rational $c$, so by (\ref{bundleexpansion}) and \cite{JM12} we have
\begin{align*}
A_{(\tY,aV_0+bV_\infty)}(\tF)=&\ord_{\tF}(V_\infty)\cdot A_{(\tY,aV_0+bV_\infty)}(V_\infty)+c\cdot A_{(\tY,aV_0+bV_\infty)}(F_{\tY})\\
=&\ord_{\tF}(V_\infty)(1-b)+c\cdot A_{(\tY,aV_0+bV_\infty)}(F_{\tY}),
\end{align*}
and
\begin{align*}
S_{(\tY,aV_0+bV_\infty)}(\tF)=&\ord_{\tF}(V_\infty)\lim_{m\to \infty}\frac{\sum_{j=0}^{m(B-A)}(m(B-A)-j)\tilde{r}_{m,j}}{\sum_{j=0}^{m(B-A)}m\tilde{r}_{m,j}}+c\cdot S_{(\tY,aV_0+bV_\infty)}(F_{\tY})\\
=& \ord_{\tF}(V_\infty)\bigg\{B-\frac{n+1}{n+2}\frac{B^{n+2}-A^{n+2}}{B^{n+1}-A^{n+1}} \bigg\}+ c\cdot S_{(\tY,aV_0+bV_\infty)}(F_{\tY}),
\end{align*}
so by Lemma \ref{delta0bundle} the following holds,
\begin{align*}
\delta_{(\tY,aV_0+bV_\infty)}(\ord_{\tF})\geq  \min\{\delta_{(\tilde{Y},aV_0+bV_\infty)}(\ord_{V_\infty}),\delta_{(\tY,aV_0+bV_\infty)}(\ord_{F_{\tY}})\}.   
\end{align*}
Combine inequality (\ref{key}), we have the following
$$\delta(\tilde{Y},aV_0+bV_\infty)\geq \min\bigg\{\delta_{(\tilde{Y},aV_0+bV_\infty)}(\ord_{V_0}),\delta_{(\tilde{Y},aV_0+bV_\infty)}(\ord_{V_\infty}), \frac{r\delta(V)}{\frac{n+1}{n+2}\frac{B^{n+2}-A^{n+2}}{B^{n+1}-A^{n+1}}}\bigg\},$$
which finishes the proof.
\end{proof}

\section{Delta invariants of projective cones of Fano type}\label{section:cone}

In this section we prove Theorem \ref{deltacone}. Similarly as bundle case in previous section, we start by the computation of $\delta_{Y,cV_\infty}(\ord_{V_0})$ and $\delta_{Y,cV_\infty}(\ord_{V_\infty})$.

\begin{lemma}\label{specialcomp}
Notation as in Theorem \ref{deltacone}, we have
$$\delta_{Y,cV_\infty}(\ord_{V_0})= \frac{(n+2)r}{(n+1)(r+1-c)}, \delta_{Y,cV_\infty}(\ord_{V_\infty})=\frac{(n+2)(1-c)}{r+1-c}$$
\end{lemma}

\begin{proof}
By Lemma \ref{conelemma} we know that 
$$K_{\tY}+V_0+cV_\infty=p^*(K_Y+cV_\infty)+rV_0, $$
hence $A_{Y,cV_\infty}(V_0)=r$ and $p^*(K_Y+cV_\infty)\sim_\bQ K_{\tY}+(1-r)V_0+cV_\infty$. Go back to the proof of Lemma \ref{delta0bundle}, we just let $a=1-r$ and $b=c$, although the values of $a,b$ may not satisfy our requirement on them anymore, it doesn't effect the computation of volumes. A direct computation concludes the lemma.
\end{proof}

By the same method used for the bundle case, we next work out another upper bound for $\delta(Y,cV_\infty)$, that is $\delta(Y,cV_\infty)\leq\frac{(n+2)r}{(n+1)(r+1-c)}\delta(V)$.

\begin{lemma}
Notation as in Theorem \ref{deltacone}, we have $\delta(Y,cV_\infty)\leq\frac{(n+2)r}{(n+1)(r+1-c)}\delta(V)$
\end{lemma}

\begin{proof}
We choose a sufficiently divisible natural number $m$ such that $-m(K_Y+cV_\infty)$ is an ample line bundle, we have
\begin{align}\label{coneexpansion}
\phi^*(-m(K_Y+cV_\infty))\sim m(r+1-c)V_\infty\sim jp^*L+(m(r+1-c)-j)V_\infty+jV_0,
\end{align}
Let $R_m$ denote $H^0(Y,-m(K_Y+cV_\infty))$, we define
$$R_{m,j}:=\{s\in R_m|\text{$\ord_{V_0}(s)=j$}\}.$$
One directly sees that
$$R_m=\bigoplus^{m(r+1-c)}_{j=0} R_{m,j}, $$
and by (\ref{coneexpansion}),
\begin{align}
R_{m,j}\cong H^0(V,jL).
\end{align}
Write $r_{m,j}=\dim R_{m,j}=\dim H^0(V,jL)$, we choose a $j$-basis type divisor $D_{m,j}$ of $L$ such that $D_{m,j}$ computes $S_{j,L}(E)$. Then we lift $D_{m,j}$ to be a divisor $\mD_{m,j}$ over $Y$. By (\ref{coneexpansion}), it is not hard to see that 
\begin{align}
\mD_m:=\frac{\sum_{j=0}^{m(r+1-c)}r_{m,j}(j\mD_{m,j}+jV_0+(m(r+1-c)-j)V_\infty)}{m\sum_{j=0}^{m(r+1-c)} r_{m,j}} 
\end{align}
 is an $m$-basis type divisor of $\phi^*(-(K_Y+cV_\infty))$. Let $E$ be a prime divisor over $V$ and $E_Y$  the divisor over $Y$ which is the natural extension of $E$, then we have 
 \begin{align}
 S_{m,(Y,cV_\infty)}(E_Y)\geq \ord_{E_Y}(\mD_m)=\frac{\sum_{j=0}^{m(r+1-c)}jr_{m,j}S_{j,L}(E)}{m\sum_{j=0}^{m(r+1-c)}r_{m,j}}.
 \end{align}
Let $m$ tends to infinity, by the same computation as in Lemma \ref{lem:computation}, one gets
\begin{align}
S_{Y,cV_\infty}(E_Y)\geq  \frac{(n+1)(r+1-c)}{n+2}\cdot S_L(E),
\end{align}
hence we have 
\begin{align}
 \delta_{Y,cV_\infty}(\ord_{E_Y})\leq \frac{n+2}{(n+1)(r+1-c)}\delta_L(\ord_E).
 \end{align}
 As $E$ is arbitrarily chosen over $V$, we know
 $$\delta(Y,cV_\infty)\leq \frac{n+2}{(n+1)(r+1-c)}\delta_V(L)=\frac{(n+2)r}{(n+1)(r+1-c)}\delta(V).$$
 \end{proof}
 
 Above all, one obtains the upper bound in Theorem \ref{deltacone}.

\begin{theorem}\label{uppercone}
Notation as in Theorem \ref{deltacone}, we have
\begin{align}
\delta(Y,cV_\infty)\leq \min\bigg\{\frac{(n+2)r}{(n+1)(r+1-c)}\delta(V), \frac{(n+2)r}{(n+1)(r+1-c)}, \frac{(n+2)(1-c)}{r+1-c}\bigg\}.
\end{align}
\end{theorem}

\begin{proof}[Proof of theorem \ref{deltacone}]
For the converse direction, one can use totally the same way as in the final proof of Theorem \ref{deltabundle} for bundle case. As there is no need to repeat it, we just leave it out.
\end{proof}

We will take another more elegant way below for the lower bound, although we can only work out the case where $0<r\leq n+1$. However, this upper bound of $r$ is usually satisfied, at least for K-semistable log Fano pairs of dimension $n$, see \cite{Liu18,Fuj18}. First we recall the following lemma \cite[Proposition 2.11]{LZ19} which establishes the relationship between K-stability of the base and that of projective cone over it, see also \cite{LL19}.

\begin{lemma}\label{conebridgelemma}
Let $(V,\Delta)$ be an n-dimensional log Fano variety, and L an ample line bundle on V such that $L\sim_\bQ -\frac{1}{r}(K_V+\Delta)$ for some $0<r\leq n+1$. As before, Y is the projective cone over V associated to L, then $(V,\Delta)$ is K-semistable if and only if $(Y,\Delta_Y+(1-\frac{r}{n+1})V_\infty)$ is K-semistable, where $\Delta_Y$ is the divisor on Y naturally extended by $\Delta$.
\end{lemma}

By above lemma, we can obtain the converse direction for projective cones.

\begin{theorem}\label{lowercone}
Notation as before, assume $0<r\leq n+1$, then
\begin{align}
\delta(Y)\geq \min\bigg\{\frac{(n+2)r}{(n+1)(r+1)}\delta(V), \frac{(n+2)r}{(n+1)(r+1)}\bigg\}.
\end{align}
\end{theorem}

\begin{proof}

We first consider the case $\delta(V)\geq 1$, i.e. $V$ is K-semistable. By above Lemma \ref{conebridgelemma}, one knows 
$\delta(Y,(1-\frac{r}{n+1})V_\infty)\geq 1$, which directly implies $\delta(Y)\geq \frac{(n+2)r}{(n+1)(r+1)}$. So in this case we conclude the result. We next deal with the case $\delta(V)<1$. Choose a sequence positive rational numbers $\delta_i$ which tends to $\delta(V)$ and $\delta_i<\delta(V)$ for each $i$. For each $\delta_i$ one can choose an effective $\bQ$-divisor $\Delta_i\sim_\bQ -K_V$ such that $(V,(1-\delta_i)\Delta_i)$ is a K-semistable log Fano pair (even uniformly K-stable), see \cite{BL18b,CRZ19}. Then we know $L\sim_\bQ -\frac{1}{r_i}(K_V+(1-\delta_i)\Delta_i)$, where $r_i=\delta_ir\leq n+1$. By Lemma \ref{conebridgelemma}, one sees $(Y,(1-\delta_i)\Delta_{i,Y}+(1-\frac{r_i}{n+1})V_\infty)$ is K-semistable, where $\Delta_{i,Y}$ is a divisor on $Y$ naturally extended by $\Delta_i$. Then we have $\delta(Y,(1-\delta_i)\Delta_{i,Y}+(1-\frac{r_i}{n+1})V_\infty)\geq 1$. It is clear that 
$$-(K_Y+(1-\delta_i)\Delta_{i,Y}+(1-\frac{r_i}{n+1})V_\infty)\sim_\bQ \frac{(n+2)r_i}{n+1}V_\infty\sim_\bQ \frac{(n+2)r_i}{(n+1)(r+1)}(-K_Y),$$
hence $\delta(Y)\geq \frac{(n+2)r_i}{(n+1)(r+1)}$. Recall $r_i=\delta_ir$, and let $i$ tends to infinity, one obtains
$$\delta(Y)\geq \frac{(n+2)r}{(n+1)(r+1)}\delta(V) ,$$ concluded.
\end{proof}

\begin{proof}[Proof of Theorem \ref{deltacone} for $0<r\leq n+1$ and $c=0$]
The proof is a combination of Theorem \ref{uppercone} and Theorem \ref{lowercone}.
\end{proof}

\begin{proof}[Proof of Corollary \ref{kssdeltacone}]
If $V$ is K-semistable, then the Fano index cannot be larger than $n+1$, hence $r\leq n+1$. So by Theorem \ref{deltacone}, $\delta(Y)= \frac{(n+2)r}{(n+1)(r+1)}$. By Lemma \ref{specialcomp}, the value is achieved by $V_0$. The proof is finished. If $V$ is not K-semistable, then by \cite[Theorem D]{BJ20}, we have
$$\delta(V)^n\cdot \vol(-K_V)\leq (n+1)^n .$$
As $-K_V\sim_\bQ rL$ and $L$ is an ample line bundle, one has $\delta(V)^nr^n\leq (n+1)^n$, which implies $r\delta(V)\leq n+1$. Hence 
$$\frac{(n+2)r}{(n+1)(r+1)}\delta(V)\leq  \frac{n+2}{r+1}.$$ 
Apply Theorem \ref{deltacone} to the case $c=0$ one directly concludes.
\end{proof}

\section{K-stability with optimal angle}\label{section:optimalangle}

In this section, as in the setting of Theorem \ref{thm:optimalangle}, we fix $V$ to be a projective Fano manifold of dimension $n$, and $S$ is a smooth divisor on $V$ such that $S\sim_\bQ -\lambda K_V$ for some positive rational number $\lambda$. Write
$$E(V,S):=\{a\in [0,1)| \textit{$(V,aS)$ is K-semistable} \}, $$
we prove the following theorem on optimal angle.

\begin{theorem}{\rm (= Theorem \ref{thm:optimalangle})}
Notation as above, suppose V and S are both K-semistable and $0<\lambda<1$, then $E(V,S)= [0,1-\frac{r}{n}]$, where $r=\frac{1}{\lambda}-1$. In particular, if $V$ and $S$ are both K-polystable, then the pair $(V,aS)$ is K-polystable for any $a\in [0, 1-\frac{r}{n})$. We also note that $1-\frac{r}{n}\geq 0$, since $\lambda\geq \frac{1}{n+1}$ if $V$ is K-semistable.

\end{theorem}

In the above theorem, we do not consider the case $\lambda\geq 1$, since in this case the K-stability of the pair $(V,aS)$ is well known to experts.  We state it here and provide a proof for the readers' convenience.

\begin{theorem}\label{geq1}
Suppose $V$ is a smooth K-semistable Fano manifold, and $S\sim_\bQ -\lambda K_V$ is a smooth divisor on V for some positive rational number $\lambda\geq 1$. Then $(V,aS)$ is K-semistable for any $a\in [0,\frac{1}{\lambda})$ and K-stable for $a\in (0,\frac{1}{\lambda})$. 
\end{theorem}

\begin{proof}
We first deal with the case $\lambda=1$, i.e. $S\sim_\bQ -K_V$, then $(V,S)$ is a log smooth Calabi-Yau pair. We show that $\alpha(V,(1-\beta)S)=1$ for sufficiently small rational $0<\beta\ll 1$ (see also \cite{Ber13}). For the definition of alpha invariant of a log Fano pair please refer to \cite{Tian87,Che01,CS08,BJ20} etc. It is clear that $\alpha(V,(1-\beta)S)\leq 1$. Suppose the inequality is strict, then we can find a divisor $D\sim_\bQ -K_V$ such that the pair $(V,(1-\beta)S+\beta D)$ is not log canonical. After subtracting certain amount of $S$ from $D$ we may assume that $S\nsubseteq\operatorname{Supp}(D)$, then by inversion of adjunction $(S,\beta D|_S)$ is not log canonical, contradicting to the choice of $\beta$ (note that $\beta$ is sufficiently small). So $(V,(1-\beta)S)$ is K-stable for sufficiently small $0<\beta\ll 1$. By interpolation of K-stability, see \cite[Proposition 2.13]{ADL19}, we conclude the result for $\lambda=1$.

For the case $\lambda>1$, it is clear that we should require $a<\frac{1}{\lambda}$ to
make sure that $(V,aS)$ is a log Fano pair. Then for any prime divisor $E$ over $V$, we have 
$$\frac{A_{V,aS}(E)}{S_{V,aS}(E)}\geq \frac{(1-a)A_V(E)}{(1-\lambda a)S_V(E)}\geq \frac{1-a}{1-\lambda a}>1$$
for any $a\in (0,\frac{1}{\lambda})$. Thus $\delta(V,aS)>1$ for any $a\in (0,\frac{1}{\lambda})$. Concluded.
\end{proof}

We turn to prove Theorem \ref{thm:optimalangle}. Let us first begin with the following lemma, see \cite[Lemma 2.12]{LZ19} or \cite[Theorem 5.1]{LZhu20}.

\begin{lemma}\label{inclusion}
Notation as in Theorem \ref{thm:optimalangle}. Suppose V and S are both K-semistable, then $(V,(1-\frac{r}{n})S)$ is K-semistable and $[0, 1-\frac{r}{n}]\subset E(V,S)$.
\end{lemma}

\begin{proof}
Write $M:= S|_S\sim_\bQ -\frac{1}{r}K_S$ to be an ample line bundle on $S$, then by \cite[Lemma 2.12]{LZ19}, $(V,(1-\frac{r}{n})S)$ can be specially degenerated to $(C_p(S, M), (1-\frac{r}{n})S_\infty)$, where $(C_p(S, M)$ is the projective cone over S associated to $M$, and $S_\infty$ is the infinite section. Since $S$ is K-semistable, by \cite[Proposition 2.11]{LZ19}, $(C_p(S, M), (1-\frac{r}{n})S_\infty)$ is K-semistable, hence $(V,(1-\frac{r}{n})S)$ is K-semistable by lower semicontinuity of delta invariants, see \cite{BL18b}. By interpolation of K-stability (\cite[Proposition 2.13]{ADL19}) we know $[0, 1-\frac{r}{n}]\subset E(V,S)$.
\end{proof}

In the above lemma, we have proved one direction of inclusion, the converse inclusion is implied by following lemma. This has been proved in \cite[Theorem 1.4]{LS14} in analytic setting, we give a purely  algebraic proof below.

\begin{lemma}\label{converseinclusion}
Notation as above, if $1-\frac{r}{n}<a<1$, then $(V,aS)$ is K-unstable.
\end{lemma}

\begin{proof}
As we have seen in the proof of Lemma \ref{inclusion},  by \cite[Lemma 2.12]{LZ19}, the pair $(V, aS)$ can be specially degenerated to $(C_p(S, M), aS_\infty)$, where $C_p(S, M)$ is the projective cone over $S$ associated to $M$, and $S_\infty$ is the infinite section. We denote this special test configuration by $\phi: (\bar{\mV},a\bar{\mS}; \bar{\mL})\to \bP^1$, where $\bar{\mL}:=-(K_{\bar{\mV}/\bP^1}+a\bar{\mS})$ is the polarization of this test configuration. Note that $-K_V\sim_\bQ (1+r)S$, so there is a $k\in\bQ$ such that $-K_{\bar{\mV}/\bP^1}\sim_\bQ (1+r)\bar{\mS}-\phi^*\mO_{\bP^1}(k)$, thus $\bar{\mL}\sim_\bQ (1+r-a)\bar{\mS}-\phi^*\mO_{\bP^1}(k)$. By \cite{LX14}, the generalised Futaki invariant (or Donaldson-Futaki invariant) of the test configuration  is as follows,
\begin{align*}
\Fut(\bar{\mV},a\bar{\mS}; \bar{\mL})
=&-\frac{1}{n+1}\frac{\bar{\mL}^{n+1}}{(-K_V-aS)^n}\\
=&-\frac{1}{n+1}\frac{\{(1+r-a)\bar{\mS}-\phi^*\mO_{\bP^1}(k)\}^{n+1}}{(-K_V-aS)^n}\\
=& k-\frac{1+r-a}{n+1}\frac{\bar{\mS}^{n+1}}{S^n}.
\end{align*}
By \cite[Lemma 2.12]{LZ19}, $\Fut(\bar{\mV},a\bar{\mS}; \bar{\mL})=0$ if $a=1-\frac{r}{n}$. By Lemma \ref{inclusion}, we know $\Fut(\bar{\mV},a\bar{\mS}; \bar{\mL})\geq 0$ when $a\leq 1-\frac{r}{n}$, so in the case $a>1-\frac{r}{n}$ we have $\Fut(\bar{\mV},a\bar{\mS}; \bar{\mL})\leq 0$. Suppose $a>1-\frac{r}{n}$, if
$\Fut(\bar{\mV},a\bar{\mS}; \bar{\mL})< 0$ then we are done;  if $\Fut(\bar{\mV},a\bar{\mS}; \bar{\mL})= 0$, then $\Fut(\bar{\mV},t\bar{\mS}; \bar{\mL})=0$ for any $t\in [0,1)$. Suppose $(V,aS)$ is K-semistable, then by \cite[Section 3]{LWX18}, the pair 
$(C_p(S,M),aS_\infty)$ is also K-semistable. By the next lemma, we know $\delta(C_p(S,M),aS_\infty)<1$ when $a>1-\frac{r}{n}$. The contradiction implies that $(V,aS)$ is K-unstable for $a>1-\frac{r}{n}$. The proof is finished.

\end{proof}

The following lemma is in fact a rephrase of Theorem \ref{deltacone}, we still state it here to fit our setting for the convenience of readers.

\begin{lemma}\label{deltaconeoverS}
Let S be a $\bQ$-Fano variety of dimension $n-1$ and $M$ an ample line bundle on S such that $M\sim_\bQ -\frac{1}{r}K_S$. Write $Y:=C_p(S,M)$ to be the projective cone over S associated to M, and $S_\infty$ the infinite section, then we have following formula computing delta invariant of $(C_p(S,M), aS_\infty)$,
$$\delta(C_p(S,M), aS_\infty)=\min \bigg\{\frac{(n+1)r}{n(r+1-a)}\delta(S),  \frac{(n+1)r}{n(r+1-a)}, \frac{(n+1)(1-a)}{r+1-a}\bigg\}. $$ 
In particular, when $a>1-\frac{r}{n}$, we have
$$\delta(C_p(S,M), aS_\infty)\leq \frac{(n+1)(1-a)}{r+1-a} <1.$$
\end{lemma}

\begin{proof}[Proof of Theorem \ref{thm:optimalangle}]
The proof is a combination of Lemma \ref{inclusion}, Lemma \ref{converseinclusion}, Lemma \ref{deltaconeoverS}, and following lemma on interpolation for K-polystability.

\end{proof}

\begin{lemma}
Let $V$ be an $n$-dimensional Fano manifold and $S$ a smooth prime divisor on $V$ such that $-K_V\sim_\bQ (1+r)S$ for some positive rational number $r$. Assume the following two conditions,
\begin{enumerate}
\item V and S are both K-polystable,
\item there is a positive rational number $0<a<1$ such that $(V,aS)$ is K-semistable,
\end{enumerate}
then we have that $(V,a'S)$ is K-polystable for any rational $0<a'<a$.
\end{lemma}

\begin{proof}
It suffices to assume $a$ to be the optimal angle for K-stability ($a=1-\frac{r}{n}$), i.e. $(V,bS)$ is K-unstable for $a<b<1$. Fix a rational number $a'\in (0,a)$. We aim to show the pair $(V,a'S)$ is K-polystable. Since it's K-semistable, by \cite[Section 3]{LWX18} there is a test configuration $(\mV, a'\mS; \mL)\to \bA^1$ such that 
\begin{enumerate}
\item $\Fut(\mV,a'\mS; \mL)=0$,
\item $(\mV_0,a'\mS_0)$ is K-polystable. 
\end{enumerate}
By interpolation of Futaki invariants (see \cite[Proposition 2.13]{ADL19}) we know $\Fut(\mV,t\mS;\mL)=0$ for any $t\in [0,1)$, which implies $\mV_0\cong V$ since $V$ is K-polystable. 
As $\Fut(\mV,a\mS;\mL)=0$, by \cite[Section 3]{LWX18} we know $ (\mV_0,a\mS_0)$ is K-semistable. By the openness property of K-polystability, cf. \cite[Proporsition 3.18]{ADL19}, 
$(V,aS)$ is strictly K-semistable (i.e. K-semistable but not K-polystable), then there is a test configuration  $(\tilde{\mV},a\tilde{\mS};\tilde{\mL})$ of $(V,aS;-(K_V+aS))$ such that
\begin{enumerate}
\item $\Fut(\tilde{\mV},a\tilde{\mS};\tilde{\mL})=0$,
\item $(\tilde{\mV}_0, a\tilde{\mS}_0)$ is K-polystable.
\end{enumerate}
Since $S$ is K-polystable, by \cite[Proporsition 2.11]{LZ19}, $(C_p(S,M), aS_\infty)$ is a K-polystable degeneration of $(V,aS)$, where $C_p(S,M)$ is the projective cone over $S$ associated to $M:=S|_S$ and $S_\infty$ denotes the infinite divisor. By the uniqueness of K-polystable degeneration \cite{BX19} we know $(\tilde{\mV}_0, a\tilde{\mS}_0)\cong (C_p(S,M),aS_\infty)$.
By \cite[Section 3]{LWX18}, there is a test configuration which has vanishing generalized Futaki invariant and degenerates $(\mV_0,a\mS_0)$ to $(\tilde{\mV}_0,a\tilde{\mS}_0)$. 
That is, we first degenerate $(V,aS)$ to $(\mV_0\cong V, a\mS_0)$, then we degenerate $(\mV_0, a\mS_0)$ to $(\tilde{\mV}_0,a\tilde{\mS}_0)\cong (C_p(S,M), aS_\infty)$. Focusing on the boundary, we first degenerate $S$ to $\mS_0$, then degenerate $\mS_0$ to $S_\infty \cong S$. By K-polystability of $S$ we at once see $\mS_0\cong S$. 
Hence, the pair $(\mV_0, a'\mS_0)$, which is the K-polystable degeneration of $(V,a'S)$, is induced by a test configuration of product type. Thus the pair $(V,a'S)$ is K-polystable. The proof is finished.
\end{proof}

\begin{corollary}{\rm (=Corollary \ref{Pncase})}
For the pair $(\bP^n, S_d)$ where $S_d$ is a smooth hypersurface of degree $1\leq d\leq n$. If $S_d$ is K-polystable (this is expected to be true), then we have $E(\bP^n, S_d)=[0, 1-\frac{r}{n}]$, where $r=\frac{n+1-d}{d}$. In particular, $(\bP^n, aS_d)$ is K-polystable for any $a\in [0, 1-\frac{r}{n})$.
\end{corollary}

\begin{proof}
Just replace $(\bP^n, S_d)$ by $(V,S)$, then apply Theorem \ref{thm:optimalangle}.
\end{proof}

\begin{example}
The pair $(\bP^2, aC)$ is K-polystable for $a\in [0, \frac{3}{4})$, where $C$ is a smooth conic curve, see \cite[Theorem 1.5]{LS14}.
\end{example}

\section{Examples}

In this section, we will compute delta invariants of some special hypersurfaces in projective space, which can be realized as the projective cones over some lower dimensional log Fano pairs.

\subsection{Projective cones over smooth Fano hypersurfaces}
Let $V^0_d\subset\bP^{n+1}$ be a smooth hypersuface of degree $d$ defined by a homogeneous polynomial $f_d$, where $2\leq d\leq n+1$. Let $H^0$ denote the hyperplane section of $V^0_d$, then $-K_{V_d^0}\sim r_0 H^0$, where $r_0=n+2-d$. Consider the projective cone over $V^0_d$ associated to $H^0$, denoted by $V^1_d$, that is $V^1_d:=\Proj \oplus_{k\in \bN}\oplus_{l\in \bN} H^0(V_d^0, kH^0) s^l$. It is clear that $V_d^1$ is still a degree $d$ hypersuface but lies in $\bP^{n+2}$, and it is still cut out by equation $f_d=0$. Write $H^1 $ the hyperplane section of $V^1_d$, then we see $-K_{V_d^2}\sim r_1 H^1$, where $r_1=n+3-d=r_0+1$. Continue the process, we get hypersurfaces $(V_d^i, H^i)$ of degree $d$ in $\bP^{n+1+i}, i\geq 1$, which is still cut out by $f_d=0$, and $H^i$ is the hyperplane section with $-K_{V^i_d}\sim r_iH^i$, where $r_i=n+2+i-d=r_0+i$. Note that $V_d^i\subset \bP^{n+1+i}, i\in \bN$ are all $\bQ$-Fano varieties. We want to compute delta invariants of these Fano varieties. Let us first recall following result.

Let $V$ be a $\bQ$-Fano variety of dimension $n$, and $L$ an ample line bundle on $V$ such that $L\sim_\bQ -\frac{1}{r}K_V$ for some positive rational number $r>0$. Let $Y$ be the projective cone over $V$ associated to $L$, that is, $Y:=\Proj \oplus_{k\in \bN}\oplus_{l\in \bN} H^0(V, kL) s^l$, then by Theorem \ref{deltacone} we know,
$$\delta(Y)=\min\bigg\{\frac{(n+2)r}{(n+1)(r+1)}\delta(V),  \frac{(n+2)r}{(n+1)(r+1)}, \frac{n+2}{r+1}\bigg\}. $$
So if $\delta(V)\leq 1$ and $r\leq n+1$, then $\delta(Y)=\frac{(n+2)r}{(n+1)(r+1)}\delta(V)$. Assume $\delta(V_d^0)\leq 1$, then
\begin{align*}
&\delta(V_d^1)=\frac{(n+2)r_0}{(n+1)(r_0+1)}\delta(V_d^0)< 1,\\
&\delta(V_d^2)=\frac{(n+3)r_1}{(n+2)(r_1+1)}\delta(V_d^1)< 1,\\
&...\\
&\delta(V_d^i)=\frac{(n+1+i)r_{i-1}}{(n+i)(r_{i-1}+1)}\delta(V_d^{i-1})< 1.
\end{align*}
Thus we have following formula

\begin{align*}
\delta(V_d^i)=\frac{(n+1+i)r_0}{(n+1)r_i}\delta(V_d^0)=\frac{(n+2-d)(n+1+i)}{(n+1)(n+2+i-d)}\delta(V_d^0)<1.
\end{align*}
Similarly, if $\delta(V_d^0)\geq 1$, then $\delta(V_d^i)=\frac{(n+2-d)(n+1+i)}{(n+1)(n+2+i-d)}<1$. We have the following results.

\begin{corollary}
Let V be a smooth hypersurface of degree d in $\bP^{n+1}$, where $2\leq d\leq n+1$, and H the hyperplane section. Then the projective cone Y over $V$ associated to H is K-unstable. In other words, $\bQ$-Fano variety cut out by the form $f_d(x_0,x_1,..., x_{n+1})=0$ in $\bP^{n+2}$ is K-unstable for $2\leq d\leq n+1$, where $f_d$ is a homogeneous form of degree d which determines a smooth Fano manifold in $\bP^{n+1}$. Moreover, 
$$\delta(Y)=\min \bigg\{ \frac{(n+2)(n+2-d)}{(n+1)(n+3-d)}\delta(V),\frac{(n+2)(n+2-d)}{(n+1)(n+3-d)}\bigg\}<1.$$
\end{corollary}

\begin{corollary}
Let V be a smooth hypersurface of degree $2\leq d\leq n+1$ in $\bP^{n+1}$, which is cut out by the form $f_d(x_0,x_1,...,x_{n+1})$, then for any positive natural number l, the variety $Y\subset \bP^{n+1+l}$ defined by $f_d(x_0,x_1,...,x_{n+1})$ is K-unstable. Moreover, 
$$\delta(Y)=\min\bigg\{\frac{(n+2-d)(n+1+l)}{(n+1)(n+2+l-d)}\delta(V), \frac{(n+2-d)(n+1+l)}{(n+1)(n+2+l-d)} \bigg\} <1.$$
\end{corollary}

\begin{example}
Let $V$ be a smooth conic curve in $\bP^2$, then $\delta(V)=1$. Let $Y$ be the projective cone over $V$, then $Y$ is isomorphic to a quadratic surface defined by $x_0^2+x_1^2+x_2^2=0$ in $\bP^3$. In this case, $n=1, d=2$, so $\delta(Y)=\frac{3}{4}$. We know that a smooth quadratic surface in $\bP^3$ is isomorphic to $\bP^1\times \bP^1$, which is K-polystable and the delta invariant is 1.
\end{example}

\begin{example}
Let $V$ be a smooth cubic surface in $\bP^3$, then $\delta(V)>1$. Let $Y$ be the projective cone over $V$, then $Y$ is isomorphic to a cubic 3-fold defined by the same $f_3(x_0,x_1,x_2,x_3)=0$ in $\bP^4$. In this case, $n=2, d=3$, so $\delta(Y)=\frac{2}{3}$. We know that a smooth cubic 3-fold in $\bP^4$ is K-stable, see \cite{LX19}.
\end{example}

\begin{question}
The delta invariants of smooth cubic surfaces has been estimated in \cite{PW18,CZ19}. It is an interesting question to estimate delta invariants of smooth cubic 3-folds and quartic 3-folds.
\end{question}

\subsection{Projective cones associated to ample $\bQ$-line bundles}

In the whole previous contents, the cones are associated to ample line bundles. We now deal with the case of ample $\bQ$-line bundles. Let $V$ be a $\bQ$-Gorenstein projective normal variety of dimension $n$ and $L$ an ample $\bQ$-line bundle on $V$. We will view $L$ to be a $\bQ$-Cartier divisor and write $\Delta_L$ to be the branched divisor associated to $L$, that is, we construct a finite morphism $f: \tilde{V}\to V$ such that $K_{\tilde{V}}=f^*(K_V+\Delta_L)$ and $f^*L$ is an ample line bundle on $\tV$. We consider the following cone over $\tV$ associated to $f^*L$. i.e.
$$Y:= \Proj\oplus_{k\in \bN}\oplus_{l\in \bN}H^0(\tV,kf^*L)s^l.$$
We assume $Y$ admits klt singularities, then $f^*L$ is proportional to $-K_{\tV}$, see \cite{Kollar13} or results in Section \ref{preliminary}, that is $\tV$ is a $\bQ$-Fano variety. Hence $(V,\Delta_L) $ is also a log Fano pair and $L$ is proportional to $-(K_V+\Delta_L)$. We assume $f^*L\sim_\bQ -\frac{1}{r}K_{\tV}$ for some positive rational $r$, then we have $L\sim_\bQ -\frac{1}{r}(K_V+\Delta_L)$. By Theorem \ref{deltacone},
$$\delta(Y)=\min \bigg\{\frac{(n+2)r}{(n+1)(r+1)}\delta(\tV),  \frac{(n+2)r}{(n+1)(r+1)}, \frac{n+2}{r+1}\bigg\} .$$ 
Combine the recent work \cite[Theorem 1.2]{LZhu20} (See also \cite{Der16b}), we have $\delta(\tV)\geq 1$ if and only if $\delta(V,\Delta_L)\geq 1$, and $\delta(\tilde{V})=\delta(V,\Delta_L)$ if either one is K-unstable. Thus we conclude the following result.

\begin{corollary}\label{qdeltabundle}
 Let $V$ be a $\bQ$-Gorenstein projective normal variety of dimension $n$ and $L$ an ample $\bQ$-line bundle. Suppose $L\sim_\bQ -\frac{1}{r}(K_V+\Delta_L)$. Write $Y:= \Proj\oplus_{k\in \bN}\oplus_{l\in \bN}H^0(V,\lfloor kL\rfloor)s^l$, then we have 
 $$\delta(Y)=\min \bigg\{\frac{(n+2)r}{(n+1)(r+1)}\delta(V,\Delta_L),  \frac{(n+2)r}{(n+1)(r+1)}, \frac{n+2}{r+1}\bigg\}. $$
\end{corollary}

\begin{example}
This example is taken from \cite{LZ19}. Let $V=\bP^n$ and $L=\frac{n+1}{n+2}S_{n+1}-nH$, where $S_{n+1}$ is a general hypersurface of degree $n+1$ in $\bP^n$ and $H$ is a hyperplane. Then $\Delta_L=\frac{n+1}{n+2}S_{n+1}$, $L\sim_\bQ \frac{1}{n+2}H$, $-(K_V+\Delta_L)\sim_\bQ \frac{n+1}{n+2}H$. Thus $r=n+1$, $\delta(V,\Delta_L)\geq 1$.
So by Theorem \ref{qdeltabundle}, $\delta(Y)=1$. In fact, $Y$ is isomorphic to a hypersurface in $\bP^{n+2}$ determined by $x_{n+1}^{n+2}=g_{n+1}x_{n+2}$, where $g_{n+1}(x_0,x_1,...,x_n)$ is the equation of $S_{n+1}\subset \bP^n$. We will deal with such kind of hypersurfaces in detail.
\end{example}

From now on, we fix $V:=\bP^n$ and $S_d\subset \bP^n$ is a smooth hypersurface of degree $d$ determined by a homogeneous form $g_d(x_0,x_1,...,x_n)$. Then we can construct a cover morphism $f: \tV\to V$, which is ramified along $S_d$ with multiplicity $k$. It is clear that we can choose $\tV$ to be a hypersurface in the weighted projective space $\bP^{n+1}(k,k,...,k,d)$ determined by $x_{n+1}^k=g_d(x_0,x_1,...,x_n)$. Suppose $H$ is the hyperplane class in $\bP^n$ and denote $L:=\frac{l}{k}S_d-\frac{dl-1}{k}H\sim_\bQ \frac{1}{k}H$, where $l$ is a positive natural number such that $l<k, (k,l)=1$ and $k| dl-1$. It is not hard to see the affine cone over $\tV$ associated to $f^*L$ is exactly the hypersurface in $\bC^{n+2}$ determined by $x_{n+1}^k=g_d(x_0,x_1,...,x_n)$, thus the corresponding projective cone is exactly the hypersurface in $\bP^{n+2}$ determined by $x_{n+1}^k=g_d\cdot x_{n+2}^{k-d}$, denoted by $Y$.
In this case, $\Delta_L=\frac{k-1}{k}S_d$, so $-(K_V+\Delta_L)\sim_\bQ (n+1-\frac{(k-1)d}{k})H\sim_\bQ ((n+1)k-(k-1)d)L$. We choose $k,d$ to satisfy that $(n+1)k-(k-1)d>0$. By Corollary \ref{qdeltabundle}, we have 
$$\delta(Y)=\min \bigg\{\frac{(n+2)r}{(n+1)(r+1)}\delta(V,\Delta_L),\frac{(n+2)r}{(n+1)(r+1)}, \frac{n+2}{r+1} \bigg\}, $$
where $r= (n+1)k-(k-1)d$. We have the following lemma about $\delta(V,\Delta_L)$.

\begin{lemma}\label{bigd}
Let $V:=\bP^n$ and $S_d\subset \bP^n$ be a smooth hypersurface of degree $d\geq n+1$, then $(V, aS_d)$ is K-semistable for any rational $0\leq a<\frac{n+1}{d}$, i.e. $\delta(V,aS_d)\geq 1$.
\end{lemma}

\begin{proof}
Apply Theorem \ref{geq1}.
\end{proof}

Above all, we have following result.

\begin{corollary}\label{deltahyper}
Let $Y\subset \bP^{n+2}$ be a hypersurface determined by 
$$x_{n+1}^k\cdot x_{n+2}^{d-k}=g_d(x_0,x_1,...,x_n),$$
where $g_d$ determines  a smooth hypersurface in $\bP^n$ of degree $n+1\leq d\leq n+2$. Suppose there is a positive natural number $l$ such that $l<k, (k,l)=1, k|dl-1$ and $(n+1)k-(k-1)d>0$, then 
$$\delta(Y)=\min \bigg\{\frac{(n+2)r}{(n+1)(r+1)}, \frac{n+2}{r+1} \bigg\}=\frac{(n+2)r}{(n+1)(r+1)}\leq 1 ,$$
where $r= (n+1)k-(k-1)d\leq  n+1$.
\end{corollary}

\begin{example}
Let $k=2, d=n+1$, and $n$ is even. Then  the hypersurface $Y\subset \bP^{n+2}$  determined by $x_{n+1}^2x_{n+2}^{n-1}=g_{n+1}(x_0,x_1,...,x_n)$ is K-semistable, where $g_{n+1}$ determines a general hypersurface of degree $n+1$ in $\bP^n$.
For example, if we take $n=2$, then the 3-dimensional hypersurface $x_3^2x_4=g_3(x_0,x_1,x_2)$ is K-semistable, where $g_3$ cuts out a smooth elliptic curve in $\bP^2$.
\end{example}

\begin{example}
Let $k=3, d=n+1$, and $3| n$ or $3| 2n+1$. Then  the hypersurface $Y\subset \bP^{n+2}$  determined by $x_{n+1}^3x_{n+2}^{n-2}=g_{n+1}(x_0,x_1,...,x_n)$ is K-semistable, where $g_{n+1}$ determines a general hypersurface of degree $n+1$ in $\bP^n$.
For example, take $n=1$, then the surface $x_2^3=g_2(x_0,x_1)x_3$ is K-semistable, where $g_2$ cuts out two different points in $\bP^1$.
\end{example}

\begin{example}
Let $k=2, d=n+2$, and $n$ is odd. Then  the hypersurface $Y\subset \bP^{n+2}$  determined by $x_{n+1}^2x_{n+2}^n=g_{n+2}(x_0,x_1,...,x_n)$ is K-unstable with $\delta(Y)=\frac{(n+2)n}{(n+1)^2}<1$, where $g_{n+2}$ determines a general hypersurface of degree $n+2$ in $\bP^n$. For example, take $n=1$, then the surface  $Y\subset \bP^3$ determined by $x_2^2x_3=g_3(x_0,x_1)$ is K-unstable with $\delta(Y)=\frac{3}{4}$, where $g_3$ cuts out 3 different points in $\bP^1$.
\end{example}

\begin{appendix}

\section{Calabi symmetry and twisted K\"ahler--Einstein edge metrics}\label{analyticpart}
\label{sec:calabi}
In this appendix we prove a weaker version of Theorem \ref{thm:delta-Y-lower-upper-bound} using purely analytic methods. More precisely, we will show
\begin{proposition}
\label{prop:weak-main-thm}
In the setting of Theorem \ref{thm:delta-Y-lower-upper-bound}, we have
\begin{equation*}
\min\bigg\{\frac{\delta(V)r\beta_0}{1+\beta_0(r-1)},\beta_0\bigg\}\leq\delta(\tilde{Y})\leq\beta_0.
\end{equation*}
\end{proposition}

We will use Calabi ansatz to carry out the proof. It is interesting to note that this approach also yields the bound obtained from the algebraic approach.
Before proceeding, we first give some necessary backgrounds for the reader's convenience.

\subsection{Calabi ansatz}
We review a well-studied and powerful construction, pioneered by Calabi \cite{Ca79,Ca82}, which can effectively produce various explicit examples of canonical metrics in K\"ahler geometry. The idea is to work on complex manifolds with certain symmetries so that one can reduce geometric PDEs to simple ODEs. This approach is often referred to as the Calabi ansatz in the literature, which has been studied and generalized to different extent by many authors; see e.g., \cite{HS98} for some general discussions and historical overviews. 

For our purpose, we will work on the total space of line bundles over K\"ahler manifolds. The goal is to construct canonical metrics on this space.
Our computation will follow the exposition in \cite[Section 4.4]{Sze14}.

Let $(V,\omega)$ be an $n$-dimensional compact K\"ahler manifold, where $\omega$ is a K\"ahler form on $V$.
Let $L\rightarrow V$ be a holomorphic line bundle equipped with a smooth Hermitian metric $h$ such that its curvature form $R_h$ satisfies
\begin{equation}
    \label{eq:Rh=lambda-omega}
    R_h:=\idd\log h^{-1}=\lambda\omega
\end{equation}
for some constant $\lambda>0.$
Let 
$$L^{-1}\xrightarrow{\pi}V$$
be the dual bundle of $L$. whose zero section will be denoted by $V_0$ (so $V_0$ is a copy of $V$ sitting inside the total space $L^{-1}$). In the following we will construct a K\"ahler metric on $L^{-1}\backslash\{V_0\}$.

The idea is to make use of the fiberwise norm on $L^{-1}$ induced by $h^{-1}$. We put
$$s(t):=\log||t||^2=\log h^{-1}(t,t),\ \text{for }t\in L^{-1}\backslash\{V_0\}.$$
So $s$ is a globally defined function on $L^{-1}\backslash\{V_0\}$. 
The goal is to construct a K\"ahler metric $\eta$ on $L^{-1}\backslash\{V_0\}$ of the form
\begin{equation}
    \label{}
    \eta=\idd f(s),
\end{equation}
where $f$ is a function to be determined.

We will carry out the computation locally. Choose $p\in V$ and let $\big(U,z=(z_1,...,z_{n})\big)$ be a local coordinate system around $p$ such that $\omega$ can be expressed by a K\"ahler potential:
\begin{equation}
    \label{eq:omega=iddbar-P}
    \omega=\idd(P(z)),
\end{equation}
where
$$P(z)=|z|^2+O(|z|^4).$$
Moreover we may assume that $L^{-1}$ is trivialized over $U$ by a nowhere vanishing holomorphic section $\sigma\in\Gamma(U,L^{-1})$ such that
$$||\sigma||_{h^{-1}}^2=h^{-1}(\sigma,\sigma)=e^{\lambda P(z)}.$$
Under this trivialization, we have an identification:
\begin{equation}
    \label{eq:def-U}
    \pi^{-1}(U)\cong U\times \mathbb{C}.
\end{equation}
Let $w$ be the holomorphic coordinate function in the fiber direction. So we have
\begin{equation}
    \label{eq:s-coord-general}
    s=\log(|w|^2e^{\lambda P(z)})\ \text{on }U\times\mathbb{C}^*.
\end{equation}
Such a choice of coordinates has the advantage that, on the fiber $\pi^{-1}(p)$ over $p$, one has
\begin{equation}
    \label{eq:dP=0-over-p}
    \partial P(z)=\overline{\partial} P(z)=0.
\end{equation}
So direct computation gives
\begin{equation}
    \label{eq:eta-of-general}
    \eta=\idd f(s)=\lambda f^\prime\pi^*\omega+f^{\prime\prime}\frac{\sqrt{-1} dw\wedge d\overline{w}}{|w|^2}.
\end{equation}
over $p$. Thus we get
\begin{equation}
    \label{eq:vol-form-of-general}
    \eta^{n+1}=\frac{(n+1) \lambda^n (f^\prime)^nf^{\prime\prime}}{|w|^2}(\pi^*\omega)^n\wedge\sqrt{-1} dw\wedge d\overline{w}.
\end{equation}
Now observe that this expression of volume form is true not just over $p$. Indeed, if we choose a different trivialization $w^\prime=q(z)w$, the expression \eqref{eq:vol-form-of-general} remains the same. So \eqref{eq:vol-form-of-general} holds everywhere on $U\times\mathbb{C}^*$.

Expression \eqref{eq:eta-of-general} indicates that, to make $\eta$ positively definite, $f$ should be a strictly convex function with $f^\prime>0$. So let us introduce
\begin{equation}
    \label{eq:def-tau-phi-general}
    \tau=f^\prime(s),\ \varphi(\tau)=f^{\prime\prime}(s).
\end{equation}
Then, over $p$, the Ricci form is given by
\begin{equation}
    \label{eq:Ric-of-general}
    \begin{aligned}
    \Ric(\eta)&=-\idd\log\operatorname{det}(\eta)\\
    &=\pi^*\Ric(\omega)-\bigg(n\lambda\frac{\varphi}{\tau}+\lambda\varphi^\prime\bigg)\pi^*\omega\\
    &\ \ \ \ -\varphi\bigg(n\frac{\varphi}{\tau}+\varphi^\prime\bigg)^\prime\frac{\sqrt{-1} dw\wedge d\overline{w}}{|w|^2}.
    \end{aligned}
    \end{equation}
Taking trace, we also get the scalar curvature:
\begin{equation}
    \label{eq:scalar}
    S(\eta)=\frac{r\pi^*S(\omega)}{\tau}-\frac{(\tau^n\varphi)^{\prime\prime}}{\tau^n}.
\end{equation}

Note that there is a natural holomorphic vector field on $L^{-1}$, which can be formulated as
\begin{equation}
    \label{eq:def-v.f.}
    v:=\nabla^{1,0}_\eta\tau.
\end{equation}
In our chosen local coordinates, simply $v=w\frac{\partial}{\partial w}$, which is in fact the vector field generated by the obvious $\bC^*$
-action on the fiber.   
    
\begin{remark}
    It is worth mentioning that, Calabi ansatz also applies to projective bundles of higher ranks (see \cite{HS98} for more general discussions).
\end{remark}

\subsection{Apply to our setting}
Assume further now we are in the setting of Theorem \ref{thm:delta-Y-lower-upper-bound}. with $\omega\in 2\pi c_1(V)$. Then we have
$$
r=\frac{1}{\lambda}.
$$
In what follows, we always assume that
\begin{equation}
    \label{eq:varp=0}
    \varphi(r-1)=\varphi(r+1)=0.
\end{equation}
For simple cohomological reason, this guarantees that $\eta$ extends to $\tilde{Y}$ and lies in $2\pi c_1(\tilde{Y})$ (but possibly singular along $V_0$ or $V_\infty$). If moreover
\begin{equation}
    \label{eq:var-smooth}
    \varphi^\prime(r-1)=1\text{ and }\varphi^\prime(r+1)=-1,
\end{equation}
then $\eta$ will indeed be a smooth K\"ahler form in $2\pi c_1(\tilde{Y})$.

\begin{lemma}
Assume that $\varphi$ satisfies \eqref{eq:varp=0} and \eqref{eq:var-smooth}, then the Futaki invariant $Fut(v)$ satisfies
$$
Fut(v)=\int_{\tilde{Y}}\tau(S(\eta)-n-1)\eta^{n+1}=(2\pi)^{n+1}(\beta_0^{-1}-1)(-K_{\tilde{Y}})^{n+1}.
$$
\end{lemma}

\begin{proof}
The first equality is exactly the definition of Futaki's invariant. While the second equality follows from a direct computation. Indeed, using the local coordinates in the previous subsection, we write $w=re^{i\theta}$. Then it is easy to deduce that
$$
\frac{\sqrt{-1}dw\wedge d\bar{w}}{|w|^2}=2\frac{dr}{r}\wedge d\theta=\frac{d\tau}{\varphi}\wedge d\theta.
$$
Thus by \eqref{eq:scalar} and \eqref{eq:vol-form-of-general},
\begin{equation*}
    \int_{\tilde{Y}}\tau(S(\eta)-n-1)\eta^{n+1}=2\pi(n+1)\int_V\int_{r-1}^{r+1}\tau\bigg(\frac{rS(\omega)}{\tau}-\frac{(\tau^n\varphi)^{\prime\prime}}{\tau^n}-n-1\bigg)\frac{\tau^n}{r^n}d\tau\wedge\omega^n.\\
\end{equation*}
Then the result follows from a tedious computation using $\int_VS(\omega)\omega^n=(2\pi)^nn(-K_V)^n$, \eqref{eq:varp=0} and \eqref{eq:var-smooth}.
\end{proof}

As a consequence, we have the following

\begin{lemma}
	\label{lem:V_0-compute-beta_0}
One has
$$
\beta(\tilde{Y})\leq\beta_0.
$$
\end{lemma}

\begin{proof}
By considering the slope of the twisted Mabuchi energy along the geodesic ray generated by $v$, 
the result follows from \cite[Proposition 7]{Sze11}.
\end{proof}

Now we derive a lower bound for the greatest Ricci lower bound $\beta(\tilde{Y})$. We will follow the approach in \cite[Section 3.1]{Sze11} to construct a family of K\"ahler metrics $\eta\in 2\pi c_1(\tilde{Y})$ with Ricci curvature as positive as possible.
Similar treatment also appears in \cite[Section 3.2]{LL19}.

Recall that $V$ is also a Fano manifold. We fix 
$$\mu\in(0,\beta(V))$$
and choose K\"ahler forms $\omega,\alpha\in2\pi c_1(V)$ such that
\begin{equation}
\label{eq:ric=mu-o+1-mu-alpha}
\Ric(\omega)=\mu\omega+(1-\mu)\alpha.
\end{equation}
Take an ample line bundle $L$ with
$$
L=-\frac{1}{r} K_V,\ \text{for some }r>1.
$$
Then using Calabi ansatz, we can consider K\"ahler metrics $\eta$ on $\tilde{Y}$ of the form (in special local coordinates)
$$
\eta=\frac{\tau}{r}\pi^*\omega+\varphi\frac{\sqrt{-1} dw\wedge d\overline{w}}{|w|^2},
$$
whose Ricci forms are given by 
\begin{equation}
\begin{aligned}
\Ric(\eta)
&=\bigg(\mu-\frac{n\varphi}{r\tau}-\frac{\varphi^\prime}{r}\bigg)\pi^*\omega+(1-\mu)\pi^*\alpha\\
&\ \ \ \ -\varphi\bigg(\frac{n\varphi}{\tau}+\varphi^\prime\bigg)^\prime\frac{\sqrt{-1} dw\wedge d\overline{w}}{|w|^2}.
\end{aligned}
\end{equation}
 Here $\varphi=\varphi(\tau)$ with $\tau\in(r-1,r+1)$ is a one-variable positive function to be determined and $w$ denotes the fiberwise coordinate. To cook up $\eta\in 2\pi c_1(\tilde{Y})$ with $\Ric(\eta)\geq\beta\eta$ (possibly in the current sense), we will impose the following conditions for $\varphi$:
\begin{equation}
\label{eq:bd-cond-for-vp-proj-bundle}
\begin{cases}
\varphi(r-1)=\varphi(r+1)=0,\\
\varphi^\prime(r-1)\in(0,1],\\
\varphi^\prime(r+1)\in[-1,0),\\
\end{cases}
\end{equation}
and
\begin{equation}
\label{eq:ode-ineq-for-vp}
-\bigg(n\frac{\varphi}{\tau}+\varphi^\prime\bigg)^\prime=\beta\text{ for }\tau\in(r-1,r+1),
\end{equation}
where $\beta$ is any constant that satisfies
\begin{equation}
\label{eq:cond-for-beta}
0<\beta\leq\min\bigg\{\frac{\mu\beta_0}{1/r+\beta_0(1-1/r)},\beta_0\bigg\}
\end{equation}
(recall the definition of $\beta_0$ in \eqref{eq:def-beta_0}).

Let us explain the exact meanings of these conditions. The boundary condition \eqref{eq:bd-cond-for-vp-proj-bundle} makes sure that $\eta\in 2\pi c_1(\tilde{Y})$ and $\eta$ possibly possesses certain amount of edge singularities along $V_0$ and $V_\infty$.
Solving the ODE \eqref{eq:ode-ineq-for-vp}, we obtain that
\begin{equation}
\tau^n\varphi=-\frac{\beta}{n+2}\tau^{n+2}+A\tau^{n+1}+B
\end{equation}
where
\begin{equation*}
\begin{cases}
A=\frac{\beta}{n+2}\cdot\frac{(r+1)^{n+2}-(r-1)^{n+2}}{(r+1)^{n+1}-(r-1)^{n+1}},\\
B=\frac{-2\beta}{n+2}\cdot\frac{(r^2-1)^{n+1}}{(r+1)^{n+1}-(r-1)^{n+1}}.\\
\end{cases}
\end{equation*}
From this, we easily derive that
\begin{equation}
\begin{cases}
\beta_1:=\varphi^\prime(r-1)=\frac{\beta}{\beta_0},\\
\beta_2:=-\varphi^\prime(r+1)=\frac{\beta(2\beta_0-1)}{\beta_0}.\\
\end{cases}
\end{equation}
Then \eqref{eq:beta_0<1} and \eqref{eq:cond-for-beta} simply imply that
$$
0<\beta_2<\beta_1\leq1.
$$
So $\eta$ has edge singularities with angles $\beta_1$ and $\beta_2$ along $V_0$ and $V_\infty$ respectively.
Moreover \eqref{eq:cond-for-beta} also guarantees that
\begin{equation}
\begin{split}
\mu-\frac{n\varphi}{r\tau}-\frac{\varphi^\prime}{r}&=\mu-\beta_1/r-\beta(1-1/r-\tau)\\
&=(\mu-\beta/(r\beta_0)-\beta(1-1/r))+\tau\beta\\
&\geq\tau\beta.\\
\end{split}
\end{equation}
Therefore $\eta$ satisfies $\Ric(\eta)\geq\beta\eta$ in the current sense. More precisely, $\eta$ solves the following twisted K\"ahler--Einstein edge equation:
\begin{equation}
\begin{split}
\Ric(\eta)&=\beta\eta+\big(\mu-\beta/(r\beta_0)-\beta(1-1/r)\big)\pi^*\omega+(1-\mu)\pi^*\alpha\\
&\ \ \ \ \ \ \ \ +2\pi(1-\beta/\beta_0)[V_0]+2\pi\big(1-\beta(2\beta_0-1)/\beta_0\big)[V_\infty].
\end{split}
\end{equation}
Then by Proposition \ref{prop:weak-main-thm}, \eqref{eq:beta_0<1}, \eqref{eq:beta-X=min-1-delta} and \cite[Theorem C]{BBJ18},
\begin{equation}
\beta(\tilde{Y})=\delta(\tilde{Y})\geq\delta_\theta(\tilde{Y})\geq\beta_\theta(\tilde{Y})\geq\beta,
\end{equation}
where
$$
\theta=\frac{(\mu-\beta/(r\beta_0)-\beta(1-1/r))}{2\pi}\pi^*\omega+\frac{1-\mu}{2\pi}\pi^*\alpha+(1-\beta_1)[V_0]+(1-\beta_2)[V_\infty]
$$
is a semi-positive current in $(1-\beta)c_1(\tilde{Y})$. Using \eqref{eq:cond-for-beta} and letting 
$\mu\rightarrow\beta(\tilde{Y})$, we obtain
$$
\beta(\tilde{Y})\geq\min\bigg\{\frac{\beta(V)\beta_0}{1/r+\beta_0(1-1/r)},\beta_0\bigg\}.
$$
Finally, using \eqref{eq:beta-X=min-1-delta} again, we get the following

\begin{proposition}
\label{prop:lower-bd}
One has
$$
\delta(\tilde{Y})\geq\min\bigg\{\frac{\delta(V)\beta_0}{1/r+\beta_0(1-1/r)},\beta_0\bigg\}.
$$
\end{proposition}
So Proposition \ref{prop:weak-main-thm} is proved.


\subsection{Limiting behavior}
Now we briefly study the degeneration of metrics on $\tilde{Y}$ with positive Ricci curvature as they approach the threshold.

	Suppose that $V$ admits a KE metric $\omega_{KE}\in 2\pi c_1(V)$. (In this case $\beta(\tilde{Y})=\beta_0$ by Theorem \ref{thm:delta-Y-lower-upper-bound}). Then as in \cite[Section 3.1]{Sze11}, for any $\beta\in(0,\beta_0)$ we can construct a \emph{smooth} K\"ahler form $\omega_\beta\in 2\pi c_1(\tilde{Y})$ with
	$\Ric(\omega_\beta)>\beta\omega_\beta$ such that, as $\beta\rightarrow\beta_0$, one has
	$
	(\tilde{Y},\omega_\beta)\xrightarrow{G.H.}(\tilde{Y},\eta),
	$
	with $\eta$ solving
	$$
	\Ric(\eta)=\beta_0\eta+(1-1/r-\beta_0(1-1/r))\pi^*\omega_{KE}+2\pi(2-2\beta_0)[V_\infty].
	$$
	In particular the limit space is still $\tilde{Y}$.
	This generalizes \cite{Sze11}, where an $\eta$ satisfying
	$$
	\Ric(\eta)=\frac{6}{7}\eta+\frac{1}{7}\pi^*\omega_{FS}+2\pi(1-\frac{5}{7})[V_\infty]
	$$
	was constructed on $Bl_1\bP^2$.
	
	Suppose in general that $V$ does not necessarily admit KE, but $\beta(V)>1/r+\beta_0(1-1/r)$. (In this case again $\beta(\tilde{Y})=\beta_0$ by Theorem \ref{thm:delta-Y-lower-upper-bound}). We choose $\mu\in(1/r+\beta_0(1-1/r),\beta(V))$ and hence there are K\"ahler forms $\omega, \alpha\in 2\pi c_1(V)$ satisfying \eqref{eq:ric=mu-o+1-mu-alpha}. Then the same construction as in \cite[Section 3.1]{Sze11} shows that, for any $\beta\in(0,\beta_0)$ there is a \emph{smooth} K\"ahler form $\omega_\beta\in 2\pi c_1(\tilde{Y})$ with
	$\Ric(\omega_\beta)>\beta\omega_\beta$ such that, as $\beta\rightarrow\beta_0$, one has
	$
	(\tilde{Y},\omega_\beta)\xrightarrow{G.H.}(\tilde{Y},\eta),
	$
	with $\eta$ solving
	$$
	\Ric(\eta)=\beta_0\eta+(\mu-1/r-\beta_0(1-1/r))\pi^*\omega+(1-\mu)\pi^*\alpha+2\pi(2-2\beta_0)[V_\infty].
	$$
	So as in the previous case, the limit space is still $\tilde{Y}$ itself. Note that the limit metric $\eta$ is not unique (as $\mu$, $\omega$ and $\alpha$ are allowed to vary).
	
	As we have noted in the above two cases, the limit space is still $\tilde{Y}$ itself. So in the view of \cite{BLZ19}, the optimal destabilization of $\tilde{Y}$ should be a product (but non-trivial) test configuration. In the toric case, similar phenomena also appeared in \cite{Li11}. We actually expect that the optimal destabilization of a toric Fano variety is always itself.

Finally, suppose that $\beta(V)\leq1/r+\beta_0(1-1/r)$. Then by Theorem \ref{thm:delta-Y-lower-upper-bound}, $\beta(\tilde{Y})=\frac{\delta(V)\beta_0)}{1/r+\beta_0(1-1/r)}$. This case turns out to be more subtle. Firstly, it seems that the Calabi ansatz does not easily provide \emph{smooth} K\"ahler forms $\omega_\beta$ such that $\Ric(\omega)\geq\beta\omega_\beta$ as $\beta\rightarrow\beta(\tilde{Y})$. Secondly, as $\mu\rightarrow\beta(X)$, the K\"ahler form $\omega$ we chose from the base $V$ (recall \eqref{eq:ric=mu-o+1-mu-alpha}) is supposed to develop certain singularities, which suggests that $V$ itself would degenerate in the Gromov--Hausdorff topology to some other $\bQ$-Fano variety. So at this stage it is unclear how $\tilde{Y}$ would degenerate. We leave this case to future studies.
\end{appendix}

\bibliography{reference.bib}

\begin{bibdiv}
\begin{biblist}

\bib{ADL19}{misc}{
      author={Ascher, Kenneth},
      author={DeVleming, Kristin},
      author={Liu, Yuchen},
       title={Wall crossing for {K}-moduli spaces of plane curves},
        date={2019},
        note={arXiv:1909.04576},
}

\bib{BBJ18}{misc}{
      author={Berman, Robert},
      author={Boucksom, S\'ebastien},
      author={Jonsson, Mattias},
       title={A variational approach to the {Yau-Tian-D}onaldson conjecture},
        date={2018},
        note={arXiv:1509.04561v2},
}

\bib{Ber13}{article}{
      author={Berman, Robert~J.},
       title={A thermodynamical formalism for {M}onge-{A}mp\`ere equations,
  {M}oser-{T}rudinger inequalities and {K}\"{a}hler-{E}instein metrics},
        date={2013},
        ISSN={0001-8708},
     journal={Adv. Math.},
      volume={248},
       pages={1254\ndash 1297},
         url={https://doi.org/10.1016/j.aim.2013.08.024},
      review={\MR{3107540}},
}

\bib{BHJ17}{article}{
      author={Boucksom, S\'{e}bastien},
      author={Hisamoto, Tomoyuki},
      author={Jonsson, Mattias},
       title={Uniform {K}-stability, {D}uistermaat-{H}eckman measures and
  singularities of pairs},
        date={2017},
        ISSN={0373-0956},
     journal={Ann. Inst. Fourier (Grenoble)},
      volume={67},
      number={2},
       pages={743\ndash 841},
         url={http://aif.cedram.org/item?id=AIF_2017__67_2_743_0},
      review={\MR{3669511}},
}

\bib{BoJ18}{misc}{
      author={Boucksom, S\'ebastien},
      author={Jonsson, Mattias},
       title={{A non-Archimedean approach to K-stability}},
        date={2018},
        note={arXiv:1805.11160},
}

\bib{BJ20}{article}{
      author={Blum, Harold},
      author={Jonsson, Mattias},
       title={Thresholds, valuations, and {K}-stability},
        date={2020},
        ISSN={0001-8708},
     journal={Adv. Math.},
      volume={365},
       pages={107062, 57},
         url={https://doi.org/10.1016/j.aim.2020.107062},
      review={\MR{4067358}},
}

\bib{BL18b}{misc}{
      author={Blum, Harold},
      author={Liu, Yuchen},
       title={Openness of uniform {K}-stability in families of
  $\mathbb{Q}$-{F}ano varieties},
        date={2020},
        note={arXiv:1808.09070},
}

\bib{BLX19}{misc}{
      author={Blum, Harold},
      author={Liu, Yuchen},
      author={Xu, Chenyang},
       title={Openness of {K-semistability for F}ano varieties},
        date={2020},
        note={arXiv:1907.02408},
}

\bib{BLZ19}{misc}{
      author={Blum, Harold},
      author={Liu, Yuchen},
      author={Zhou, Chuyu},
       title={Optimal destabilization of {K-unstable F}ano varieties via
  stability thresholds},
        date={2020},
        note={arXiv:1907.05399},
}

\bib{BX19}{article}{
      author={Blum, Harold},
      author={Xu, Chenyang},
       title={Uniqueness of {K}-polystable degenerations of {F}ano varieties},
        date={2019},
        ISSN={0003-486X},
     journal={Ann. of Math. (2)},
      volume={190},
      number={2},
       pages={609\ndash 656},
         url={https://doi.org/10.4007/annals.2019.190.2.4},
      review={\MR{3997130}},
}

\bib{Ca79}{article}{
      author={Calabi, E.},
       title={M\'{e}triques k\"{a}hl\'{e}riennes et fibr\'{e}s holomorphes},
        date={1979},
        ISSN={0012-9593},
     journal={Ann. Sci. \'{E}cole Norm. Sup. (4)},
      volume={12},
      number={2},
       pages={269\ndash 294},
         url={http://www.numdam.org/item?id=ASENS_1979_4_12_2_269_0},
      review={\MR{543218}},
}

\bib{Ca82}{article}{
      author={Calabi, Eugenio},
       title={Extremal {K}\"{a}hler metrics},
        date={1982},
      volume={102},
       pages={259\ndash 290},
      review={\MR{645743}},
}

\bib{Che01}{article}{
      author={Cheltsov, I.~A.},
       title={Log canonical thresholds on hypersurfaces},
        date={2001},
        ISSN={0368-8666},
     journal={Mat. Sb.},
      volume={192},
      number={8},
       pages={155\ndash 172},
         url={https://doi.org/10.1070/SM2001v192n08ABEH000591},
      review={\MR{1862249}},
}

\bib{CRZ19}{article}{
      author={Cheltsov, Ivan~A.},
      author={Rubinstein, Yanir~A.},
      author={Zhang, Kewei},
       title={Basis log canonical thresholds, local intersection estimates, and
  asymptotically log del {P}ezzo surfaces},
        date={2019},
        ISSN={1022-1824},
     journal={Selecta Math. (N.S.)},
      volume={25},
      number={2},
       pages={25:34},
         url={https://doi.org/10.1007/s00029-019-0473-z},
      review={\MR{3945265}},
}

\bib{CS08}{article}{
      author={Cheltsov, I.~A.},
      author={Shramov, K.~A.},
       title={Log-canonical thresholds for nonsingular {F}ano threefolds},
        date={2008},
        ISSN={0042-1316},
     journal={Uspekhi Mat. Nauk},
      volume={63},
      number={5(383)},
       pages={73\ndash 180},
         url={https://doi.org/10.1070/RM2008v063n05ABEH004561},
      review={\MR{2484031}},
}

\bib{CZ19}{article}{
      author={Cheltsov, Ivan},
      author={Zhang, Kewei},
       title={Delta invariants of smooth cubic surfaces},
        date={2019},
        ISSN={2199-675X},
     journal={Eur. J. Math.},
      volume={5},
      number={3},
       pages={729\ndash 762},
         url={https://doi.org/10.1007/s40879-019-00357-0},
      review={\MR{3993261}},
}

\bib{Der16b}{article}{
      author={Dervan, Ruadha\'{\i}},
       title={On {K}-stability of finite covers},
        date={2016},
        ISSN={0024-6093},
     journal={Bull. Lond. Math. Soc.},
      volume={48},
      number={4},
       pages={717\ndash 728},
         url={https://doi.org/10.1112/blms/bdw029},
      review={\MR{3532146}},
}

\bib{FO18}{article}{
      author={Fujita, Kento},
      author={Odaka, Yuji},
       title={On the {K}-stability of {F}ano varieties and anticanonical
  divisors},
        date={2018},
        ISSN={0040-8735},
     journal={Tohoku Math. J. (2)},
      volume={70},
      number={4},
       pages={511\ndash 521},
         url={https://doi.org/10.2748/tmj/1546570823},
      review={\MR{3896135}},
}

\bib{Fuj18}{article}{
      author={Fujita, Kento},
       title={Optimal bounds for the volumes of {K}\"{a}hler-{E}instein {F}ano
  manifolds},
        date={2018},
        ISSN={0002-9327},
     journal={Amer. J. Math.},
      volume={140},
      number={2},
       pages={391\ndash 414},
         url={https://doi.org/10.1353/ajm.2018.0009},
      review={\MR{3783213}},
}

\bib{Fuj19c}{article}{
      author={Fujita, Kento},
       title={A valuative criterion for uniform {K}-stability of {$\Bbb
  Q$}-{F}ano varieties},
        date={2019},
        ISSN={0075-4102},
     journal={J. Reine Angew. Math.},
      volume={751},
       pages={309\ndash 338},
         url={https://doi.org/10.1515/crelle-2016-0055},
      review={\MR{3956698}},
}

\bib{Golota19}{article}{
      author={Golota, Aleksei},
       title={Delta-invariants for {F}ano varieties with large automorphism
  groups},
        date={2020},
        ISSN={0129-167X},
     journal={Internat. J. Math.},
      volume={31},
      number={10},
       pages={2050077, 31},
         url={https://doi.org/10.1142/S0129167X20500779},
      review={\MR{4153354}},
}

\bib{HS98}{article}{
      author={Hwang, Andrew~D.},
      author={Singer, Michael~A.},
       title={A momentum construction for circle-invariant {K}\"{a}hler
  metrics},
        date={2002},
        ISSN={0002-9947},
     journal={Trans. Amer. Math. Soc.},
      volume={354},
      number={6},
       pages={2285\ndash 2325},
         url={https://doi.org/10.1090/S0002-9947-02-02965-3},
      review={\MR{1885653}},
}

\bib{JM12}{article}{
      author={Jonsson, Mattias},
      author={Musta\c{t}\u{a}, Mircea},
       title={Valuations and asymptotic invariants for sequences of ideals},
        date={2012},
        ISSN={0373-0956},
     journal={Ann. Inst. Fourier (Grenoble)},
      volume={62},
      number={6},
       pages={2145\ndash 2209 (2013)},
         url={https://doi.org/10.5802/aif.2746},
      review={\MR{3060755}},
}

\bib{KM98}{book}{
      author={Koll\'{a}r, J\'{a}nos},
      author={Mori, Shigefumi},
       title={Birational geometry of algebraic varieties},
      series={Cambridge Tracts in Mathematics},
   publisher={Cambridge University Press, Cambridge},
        date={1998},
      volume={134},
        ISBN={0-521-63277-3},
         url={https://doi.org/10.1017/CBO9780511662560},
        note={With the collaboration of C. H. Clemens and A. Corti, Translated
  from the 1998 Japanese original},
      review={\MR{1658959}},
}

\bib{Kollar13}{book}{
      author={Koll\'{a}r, J\'{a}nos},
       title={Singularities of the minimal model program},
      series={Cambridge Tracts in Mathematics},
   publisher={Cambridge University Press, Cambridge},
        date={2013},
      volume={200},
        ISBN={978-1-107-03534-8},
         url={https://doi.org/10.1017/CBO9781139547895},
        note={With a collaboration of S\'{a}ndor Kov\'{a}cs},
      review={\MR{3057950}},
}

\bib{Kollar96}{book}{
      author={Koll\'{a}r, J\'{a}nos},
       title={Rational curves on algebraic varieties},
      series={Ergebnisse der Mathematik und ihrer Grenzgebiete. 3. Folge. A
  Series of Modern Surveys in Mathematics [Results in Mathematics and Related
  Areas. 3rd Series. A Series of Modern Surveys in Mathematics]},
   publisher={Springer-Verlag, Berlin},
        date={1996},
      volume={32},
        ISBN={3-540-60168-6},
         url={https://doi.org/10.1007/978-3-662-03276-3},
      review={\MR{1440180}},
}

\bib{Li11}{article}{
      author={Li, Chi},
       title={Greatest lower bounds on {R}icci curvature for toric {F}ano
  manifolds},
        date={2011},
        ISSN={0001-8708},
     journal={Adv. Math.},
      volume={226},
      number={6},
       pages={4921\ndash 4932},
         url={https://doi.org/10.1016/j.aim.2010.12.023},
      review={\MR{2775890}},
}

\bib{Liu18}{article}{
      author={Liu, Yuchen},
       title={The volume of singular {K}\"{a}hler-{E}instein {F}ano varieties},
        date={2018},
        ISSN={0010-437X},
     journal={Compos. Math.},
      volume={154},
      number={6},
       pages={1131\ndash 1158},
         url={https://doi.org/10.1112/S0010437X18007042},
      review={\MR{3797604}},
}

\bib{LL19}{article}{
      author={Li, Chi},
      author={Liu, Yuchen},
       title={K\"{a}hler-{E}instein metrics and volume minimization},
        date={2019},
        ISSN={0001-8708},
     journal={Adv. Math.},
      volume={341},
       pages={440\ndash 492},
         url={https://doi.org/10.1016/j.aim.2018.10.038},
      review={\MR{3872852}},
}

\bib{Li2020}{book}{
      author={Li, Chi},
      author={Liu, Yuchen},
      author={Xu, Chenyang},
       title={A guided tour to normalized volume},
   publisher={Springer International Publishing},
     address={Cham},
        date={2020},
}

\bib{LS14}{article}{
      author={Li, Chi},
      author={Sun, Song},
       title={Conical {K}\"{a}hler-{E}instein metrics revisited},
        date={2014},
        ISSN={0010-3616},
     journal={Comm. Math. Phys.},
      volume={331},
      number={3},
       pages={927\ndash 973},
         url={https://doi.org/10.1007/s00220-014-2123-9},
      review={\MR{3248054}},
}

\bib{LWX18}{misc}{
      author={Li, Chi},
      author={Wang, Xiaowei},
      author={Xu, Chenyang},
       title={Algebraicity of the metric tangent cones and equivariant
  {K}-stability},
        date={2021},
        note={arXiv:1805.03393},
}

\bib{LX14}{article}{
      author={Li, Chi},
      author={Xu, Chenyang},
       title={Special test configuration and {K}-stability of {F}ano
  varieties},
        date={2014},
        ISSN={0003-486X},
     journal={Ann. of Math. (2)},
      volume={180},
      number={1},
       pages={197\ndash 232},
         url={https://doi.org/10.4007/annals.2014.180.1.4},
      review={\MR{3194814}},
}

\bib{LX19}{article}{
      author={Liu, Yuchen},
      author={Xu, Chenyang},
       title={K-stability of cubic threefolds},
        date={2019},
        ISSN={0012-7094},
     journal={Duke Math. J.},
      volume={168},
      number={11},
       pages={2029\ndash 2073},
         url={https://doi.org/10.1215/00127094-2019-0006},
      review={\MR{3992032}},
}

\bib{LZ19}{article}{
      author={{Liu}, Yuchen},
      author={{Zhuang}, Ziquan},
       title={{On the sharpness of Tian's criterion for K-stability}},
        date={2019Mar},
     journal={arXiv e-prints},
       pages={arXiv:1903.04719},
      eprint={1903.04719},
}

\bib{LZhu20}{misc}{
      author={Liu, Yuchen},
      author={Zhu, Ziwen},
       title={Equivariant k-stability under finite group action},
        date={2020},
        note={arXiv:2001.10557},
}

\bib{PW18}{article}{
      author={Park, Jihun},
      author={Won, Joonyeong},
       title={K-stability of smooth del {P}ezzo surfaces},
        date={2018},
        ISSN={0025-5831},
     journal={Math. Ann.},
      volume={372},
      number={3-4},
       pages={1239\ndash 1276},
         url={https://doi.org/10.1007/s00208-017-1602-7},
      review={\MR{3880298}},
}

\bib{Rub08}{article}{
      author={Rubinstein, Yanir~A.},
       title={Some discretizations of geometric evolution equations and the
  {R}icci iteration on the space of {K}\"{a}hler metrics},
        date={2008},
        ISSN={0001-8708},
     journal={Adv. Math.},
      volume={218},
      number={5},
       pages={1526\ndash 1565},
         url={https://doi.org/10.1016/j.aim.2008.03.017},
      review={\MR{2419932}},
}

\bib{Rub09}{article}{
      author={Rubinstein, Yanir~A.},
       title={On the construction of {N}adel multiplier ideal sheaves and the
  limiting behavior of the {R}icci flow},
        date={2009},
        ISSN={0002-9947},
     journal={Trans. Amer. Math. Soc.},
      volume={361},
      number={11},
       pages={5839\ndash 5850},
         url={https://doi.org/10.1090/S0002-9947-09-04675-3},
      review={\MR{2529916}},
}

\bib{Sze11}{article}{
      author={Sz\'{e}kelyhidi, G\'{a}bor},
       title={Greatest lower bounds on the {R}icci curvature of {F}ano
  manifolds},
        date={2011},
        ISSN={0010-437X},
     journal={Compos. Math.},
      volume={147},
      number={1},
       pages={319\ndash 331},
         url={https://doi.org/10.1112/S0010437X10004938},
      review={\MR{2771134}},
}

\bib{Sze14}{book}{
      author={Sz\'{e}kelyhidi, G\'{a}bor},
       title={An introduction to extremal {K}\"{a}hler metrics},
      series={Graduate Studies in Mathematics},
   publisher={American Mathematical Society, Providence, RI},
        date={2014},
      volume={152},
        ISBN={978-1-4704-1047-6},
      review={\MR{3186384}},
}

\bib{Tian87}{article}{
      author={Tian, Gang},
       title={On {K}\"{a}hler-{E}instein metrics on certain {K}\"{a}hler
  manifolds with {$C_1(M)>0$}},
        date={1987},
        ISSN={0020-9910},
     journal={Invent. Math.},
      volume={89},
      number={2},
       pages={225\ndash 246},
         url={https://doi.org/10.1007/BF01389077},
      review={\MR{894378}},
}

\bib{Tia92}{article}{
      author={Tian, Gang},
       title={On stability of the tangent bundles of {F}ano varieties},
        date={1992},
        ISSN={0129-167X},
     journal={Internat. J. Math.},
      volume={3},
      number={3},
       pages={401\ndash 413},
         url={https://doi.org/10.1142/S0129167X92000175},
      review={\MR{1163733}},
}

\bib{zhuang20}{article}{
      author={Zhuang, Ziquan},
       title={Product theorem for {K}-stability},
        date={2020},
        ISSN={0001-8708},
     journal={Adv. Math.},
      volume={371},
       pages={107250, 18},
         url={https://doi.org/10.1016/j.aim.2020.107250},
      review={\MR{4108221}},
}

\end{biblist}
\end{bibdiv}
\end{document}